\documentclass[10pt,a4paper]{amsart}
\usepackage{amsmath, amssymb, amsfonts, amsthm, float, stmaryrd, epsfig}
\usepackage[pdfborder={0 0 0 [0 0 ]},bookmarksdepth=3]{hyperref}
\usepackage[latin1]{inputenc}
\usepackage[capitalize]{cleveref}
\usepackage{color}
\usepackage[ps,all,arc,rotate]{xy}
\usepackage{bbm} 

\hypersetup{
    colorlinks=true,
    linkcolor=black,
    citecolor=black,
    filecolor=black,
    urlcolor=black,
}

\numberwithin{figure}{section}
\numberwithin{table}{section}

\theoremstyle{plain}
\newtheorem{thm}{Theorem}[section]
\crefname{thm}{Theorem}{Theorems}
\newtheorem*{prop*}{Proposition}
\newtheorem*{thm*}{Theorem}
\newtheorem{prop}[thm]{Proposition}
\crefname{prop}{Proposition}{Propositions}
\newtheorem{lem}[thm]{Lemma}
\crefname{lem}{Lemma}{Lemmata}
\newtheorem{cor}[thm]{Corollary}
\crefname{cor}{Corollary}{Corollaries}

\theoremstyle{definition}

\newtheorem{dfn}[thm]{Definition}
\newtheorem*{dfn*}{Definition}

\theoremstyle{remark}
\newtheorem{rmk}[thm]{Remark}

\newtheoremstyle{maintheorem}{}{}{\itshape}{}{\bfseries}{}{.5em}{#1 \!\thmnote{#3}.}
\theoremstyle{maintheorem}
\newtheorem*{mainthm}{Theorem}

\makeatletter
\let\c@figure\c@thm
\let\c@table\c@thm
\makeatother
\crefname{figure}{Figure}{Figures}
\crefname{table}{Table}{Tables}

\newcommand{\Aut}{\operatorname{Aut}}
\newcommand{\SAut}{\operatorname{SAut}}
\newcommand{\Out}{\operatorname{Out}}
\newcommand{\SOut}{\operatorname{SOut}}
\newcommand{\GL}{\operatorname{GL}}
\newcommand{\SL}{\operatorname{SL}}

\newcommand{\PGL}{\operatorname{PGL}}
\newcommand{\PSL}{\operatorname{PSL}}
\renewcommand{\L}{\operatorname{L}}

\newcommand{\im}{\operatorname{im}}

\newcommand{\Alt}{\operatorname{Alt}\nolimits}
\newcommand{\Sym}{\operatorname{Sym}\nolimits}

\newcommand{\arank}{$A$-rank}
\newcommand{\drank}{$D'$-rank}

\newcommand{\A}{\operatorname{\mathtt{A}}}
\newcommand{\Atwo}{\operatorname{{^2\hspace{-1 pt}}\mathtt{A}}}
\newcommand{\B}{\operatorname{\mathtt{B}}}
\newcommand{\Btwo}{\operatorname{{^2}\mathtt{B}}}
\newcommand{\typeC}{\operatorname{\mathtt{C}}}
\newcommand{\D}{\operatorname{\mathtt{D}}}
\newcommand{\Dtwo}{\operatorname{{^2}\mathtt{D}}}
\newcommand{\Dthree}{\operatorname{{^3}\mathtt{D}}}
\newcommand{\E}{\operatorname{\mathtt{E}}}
\newcommand{\Etwo}{\operatorname{{^2}\mathtt{E}}}
\newcommand{\typeF}{\operatorname{\mathtt{F}}}
\newcommand{\typeFtwo}{\operatorname{{^2}\mathtt{F}}}
\newcommand{\G}{\operatorname{\mathtt{G}}}
\newcommand{\Gtwo}{\operatorname{{^2}\mathtt{G}}}

\newcommand{\Fi}{\operatorname{Fi}}

\def\N{\mathbb{N}}
\def\Z{\mathbb{Z}}

\def\1{\mathbbm{1}}

\def\F{\mathbb{F}}
\def\Fbar{\overline{\mathbb{F}}}

\def\s-{\smallsetminus}
\def\into{\hookrightarrow}

\newcounter{dawidcomments}

\newcounter{emiliocomments}

\newcounter{barbaracomments}

\author{Barbara Baumeister, Dawid Kielak, and Emilio Pierro}
\title{On the smallest non-abelian quotient of $\Aut(F_n)$}
\date{\today}

\begin{document}

\begin{abstract}
We show that the smallest non-abelian quotient of $\Aut(F_n)$ is $\PSL_n(\Z/2\Z) = \L_n(2)$, thus confirming a conjecture of Mecchia--Zimmermann. In the course of the proof we give an exponential (in $n$) lower bound for the cardinality of a set on which $\SAut(F_n)$, the unique index $2$ subgroup of $\Aut(F_n)$, can act non-trivially. We also offer new results on the representation theory of $\SAut(F_n)$ in small dimensions over small, positive characteristics, and on rigidity of maps from $\SAut(F_n)$ to finite groups of Lie type and algebraic groups in characteristic $2$.
\end{abstract}

\maketitle

\begin{center}
\emph{Dedicated to the memory of Kay Magaard.}
\end{center}

\section{Introduction}
\label{sec: intro}

In \cite{Lyndon1987} Lyndon stated (denoting a free group by $F$):
\begin{quote}
 Problem 5. Determine the structure of $\Aut(F)$, of its subgroups, especially its finite subgroups, and its quotient groups, as well as the structure of individual automorphisms.
\end{quote}
This (admittedly very broad) question, and its variation for $\Out(F_n)$, the group of outer automorphisms of the free group $F_n$, can be seen as one of the driving forces behind much activity in Geometric Group Theory. The structure of finite subgroups has been understood completely thanks to the Nielsen realisation theorems for $\Aut(F_n)$ and $\Out(F_n)$ (proven independently by Culler~\cite{Culler1984}, Khramtsov~\cite{Khramtsov1985}, and Zimmermann~\cite{Zimmermann1981}). In addition,  there is a rich literature on both the structure of general subgroups and individual elements of $\Aut(F_n)$ and $\Out(F_n)$ -- see e.g. the series of papers of Handel--Mosher~\cite{HandelMosher2013a,HandelMosher2013b,HandelMosher2013c,HandelMosher2013d,HandelMosher2013e} and the work of Bestvina--Feighn--Handel~\cite{Bestvinaetal2000,Bestvinaetal2005} on the train track theory.

In this article we focus on quotients of $\Aut(F_n)$, particularly the finite ones.
It should be noted that such investigations, being closely related to the study of finite index subgroups of $\Aut(F_n)$, are connected to one of the most important open problems relating to automorphisms of free groups: for $n\geqslant 3$, does $\Aut(F_n)$ or $\Out(F_n)$ have Kazhdan's property (T)?
These questions have negative answers for $n=3$ (as shown by Grunewald--Lubotzky~\cite{GrunewaldLubotzky2006}, and positive answers for $n=5$ (as shown by Kaluba--Nowak--Ozawa~\cite{Kalubaetal2017}).

It is in fact open whether either of $\Aut(F_n)$ or $\Out(F_n)$ (for $n = 4$ or $n \geqslant 6$) contains a finite index subgroup which can map onto $F_2$, or indeed onto $\Z$. (The former property is \emph{largeness}, the latter is the negation of the property \emph{FAb}.) Positive answers to either of these questions  would lead to the negation of property (T).

More directly related to the current article is the question of whether $\Aut(F_n)$ satisfies the \emph{Congruence Subgroup Property} -- in this context, the property says that every finite quotient $\Aut(F_n) \to K$ should factor through a (finite) quotient $\Aut(F_n) \to \Aut(F_n/\chi)$, where $\chi$ is a finite index characteristic subgroup of $F_n$. The Congruence Subgroup Property was shown to hold for $n=2$ by Asada~\cite{Asada2001} (see also~\cite{Buxetal2011} for a translation of the proof into more group-theoretic terms); the corresponding statement for $\SL_n(\Z)$ is true for all $n>2$ (as shown by Mennicke~\cite{Mennicke1965}), while  the statement for mapping class groups remains  open.

\smallskip
Investigating finite quotients of outer automorphism groups of free groups, as well as mapping class groups, has a long history. The first fundamental result here is that groups in both classes are residually finite -- this is due to Grossman~\cite{Grossman1974}.
Once we know that the groups admit many finite quotients, we can start asking questions about the structure or size of such quotients. This is of course equivalent to studying normal subgroups of finite index in mapping class groups and $\Out(F_n)$.

When $n\geqslant 3$, the groups $\Out(F_n)$ and $\Aut(F_n)$ have unique subgroups of index $2$, denoted respectively by $\SOut(F_n)$ and $\SAut(F_n)$. Both of these subgroups are perfect, and
so the abelian quotients of $\Out(F_n)$ and $\Aut(F_n)$ are the two groups of order at most $2$. The situation for mapping class groups is very similar.

The simplest way of obtaining a non-abelian quotient of $\Out(F_n)$ or $\Aut(F_n)$ comes from observing that $\Out(F_n)$ acts on the abelianisation of $F_n$, that is $\Z^n$. In this way we obtain (surjective) maps
\[
\Aut(F_n) \to \Out(F_n) \to \GL_n(\Z)                                                                                                                                                                                                                                                                \]
The finite quotients of $\GL_n(\Z)$ are controlled by the Congruence Subgroup Property and are well understood. In particular, the smallest (in terms of cardinality) such quotient is $\PSL_n(\Z/2\Z) = \L_n(2)$, obtained by reducing $\Z$ modulo $2$. According to a conjecture of Mecchia--Zimmermann \cite{MecchiaZimmermann2010}, the group $\L_n(2)$ is the smallest non-abelian quotient of $\Out(F_n)$.

In~\cite{MecchiaZimmermann2010} Mecchia and Zimmermann confirmed their conjecture for $n \in \{3,4\}$. In this paper we prove it for all $n \geqslant 3$. In fact we prove more:

\begin{mainthm}[\ref{very main thm}]
 Let $n\geqslant 3$. Every non-trivial finite quotient of $\SAut(F_n)$ is either greater in cardinality than $\L_n(2)$, or isomorphic to $\L_n(2)$. Moreover,  if the quotient is $\L_n(2)$, then the quotient map is the natural map postcomposed with an automorphism of $\L_n(2)$.
\end{mainthm}
The natural map $\SAut(F_n) \to \L_n(2)$ is obtained by acting on $\mathrm{H}_1(F_n;\Z/2\Z)$. Note that this result confirms the Congruence Subgroup Property for the minimal quotients of $\SAut(F_n)$.

Zimmermann~\cite{Zimmermann2012} also
formulated a corresponding conjecture for mapping class groups; this has now been solved by the second- and third-named authors in \cite{KielakPierro2017}.

Let us remark here that, even though the main result deals specifically with the smallest non-trivial quotient of $\SAut(F_n)$, the techniques developed in this paper can be used to study other finite quotients and so yield information on normal finite index subgroups of $\SAut(F_n)$ in general.

\smallskip

In order to determine the smallest non-trivial quotient of $\SAut(F_n)$  we can restrict our attention to the finite simple groups which, by the Classification of Finite Simple Groups (CFSG), fall into one of the following four families:\begin{enumerate}
\item the cyclic groups of prime order;
\item the alternating groups $A_n$, for $n \geqslant 5$;
\item the finite groups of Lie type, and;
\item the $26$ sporadic groups.
\end{enumerate}
For the full statement of the CFSG we refer the reader to \cite[Chapter 1]{atlas} and for a more detailed exposition of the non-abelian finite simple groups to \cite{raw}. For the purpose of this paper, we further divide the finite groups of Lie type into the following two families:
\begin{enumerate}
\item[($3$C)] the ``classical groups'': $\A_n$, $\Atwo_n$, $\B_n$, $\typeC_n$, $\D_n$ and $\Dtwo_n$, and;
\item[($3$E)] the ``exceptional groups'': $\Btwo_2$, $\Gtwo_2$, $\typeFtwo_4$, $\Dthree_4$, $\Etwo_6$, $\G_2$, $\typeF_4$, $\E_6$, $\E_7$ and $\E_8$.
\end{enumerate}

We turn first to the alternating groups and prove the following.

\begin{mainthm}[\ref{main thm actions}]
Let $n \geqslant 3$. Any action of $\SAut(F_n)$ on a set with fewer than $k(n)$ elements is trivial, where
\[
k(n) = \left\{ \begin{array}{ccl}
7 & & n=3 \\
8 & & n=4 \\
12 & \textrm{ if } & n=5 \\
14 &  & n=6 
\end{array} \right.
\]
and $k(n) = \max \limits_{r \leqslant \frac n 2 -3} \min \{ 2^{n-r-p(n)}, \binom n r \}$ for $n\geqslant 7$,
where $p(n)$ equals $0$ when $n$ is odd and $1$ when $n$ is even.
\end{mainthm}
The bound given above for $n\geqslant 7$ is somewhat mysterious; one can however easily see that (for large $n$) it is bounded below by $2^{\frac n 2}$.

Note that, so far, no such result was available for $\SAut(F_n)$ (one could extract a bound of $2n$ from the work of Bridson--Vogtmann~\cite{BridsonVogtmann2003}). Clearly, the bounds given above give precisely the same bounds for $\SOut(F_n)$. In this context the best bound known so far was $\frac 1 2 \binom {n+1} 2$ (for $n\geqslant 6$). It was obtained by the second-named author in~\cite[Corollary 2.24]{Kielak2016} by an argument of representation theoretic flavour. The proof contained in the current paper is more direct.

The question of the smallest set on which $\SAut(F_n)$ or $\SOut(F_n)$ can act non-trivially remains open, but we do answer the question on the growth of the size of such a set with $n$ -- it is exponential. Note that the corresponding question for mapping class groups has been answered by Berrick--Gebhardt--Paris~\cite{Berricketal2014}.

Let us remark here that $\Out(F_n)$ (and hence also $\SAut(F_n)$) has plenty of alternating quotients -- indeed, it was shown by Gilman~\cite{Gilman1977} that $\Out(F_n)$ is residually alternating.

We use the bounds above to improve on a previous result of the second-named author on rigidity of outer actions of $\Out(F_n)$ on free groups (see \cref{thm phd dawid} for details).

\smallskip
Following the alternating groups, we rule out the sporadic groups. It was observed by Bridson--Vogtman~\cite{BridsonVogtmann2003} that any quotient of $\SAut(F_n)$ which does not factor through $\SL_n(\Z)$ must contain a subgroup isomorphic to
\[(\Z / 2\Z)^{n-1} \rtimes A_n = 2^{n-1} \rtimes A_n\]
(one can easily see this subgroup inside of $\SAut(F_n)$, as it acts on the $n$-rose, that is the bouquet of $n$ circles).
Thus, for large enough $n$, sporadic groups are never quotients of $\SAut(F_n)$, and therefore our proof (asymptotically) is not sensitive to whether the list of sporadic groups is really complete.

\smallskip
Finally we turn to the finite groups of Lie type. Our strategy differs depending on whether we are dealing with the classical or exceptional groups. The exceptional groups are handled in a similar fashion to the sporadic groups
-- this time we use an alternating subgroup $A_{n+1}$ inside $\SAut(F_n)$, which rigidifies the group in a similar way as the subgroup $2^{n-1} \rtimes A_n$ did.
The degrees of the largest alternating subgroups of exceptional groups of Lie type are  determined in \cite{lieseitz};
in particular this degree is bounded above by 17 across all such groups.

\smallskip
The most involved part of the paper deals with the classical groups. In characteristic $2$ we use an inductive strategy, and prove
\begin{mainthm}[\ref{main thm char 2 lie type}]
Let $n \geqslant 3$. Let $K$ be a finite
group of Lie type in characteristic $2$ of twisted rank less than $n-1$, and let $\overline K$ be a reductive algebraic group over an algebraically closed field of characteristic $2$ of rank less than $n-1$. Then any homomorphism $\Aut(F_n) \to K$ or $\Aut(F_n) \to \overline K$ has abelian image, and any homomorphism $\SAut(F_{n+1}) \to K$ or $\SAut(F_{n+1}) \to \overline K$ is trivial.
\end{mainthm}
Recall that there are precisely two abelian quotients of $\Aut(F_n)$ (when $n \geqslant 3$), namely $\Z/2\Z=2$ and the trivial group.

In odd characteristic we need to investigate the representation theory of $\SAut(F_n)$.
We prove
\begin{mainthm}[\ref{main thm: representations}]
Let $n\geqslant 8$.
Every irreducible projective representation of $\SAut(F_n)$ of dimension less than $2n-4$ over a field of characteristic greater than $2$ which does not factor through the natural map $\SAut(F_n) \to \L_n(2)$ has dimension $n+1$.
\end{mainthm}
Note that over a field of characteristic greater than $n+1$, every linear representation of $\Out(F_n)$ of dimension less than $\binom {n+1} 2$ factors through the map $\Out(F_n) \to \GL_n(\Z)$ mentioned above (see~\cite[3.13]{Kielak2013}).
Representations of $\SAut(F_n)$ over characteristic other than $2$ have also been studied by Varghese~\cite{Varghese2015}.

\smallskip
The proof of the main result (\cref{very main thm}) for $n \geqslant 8$ is uniform; the small values of $n$ need special attention, and we deal with them at the end of the paper. We also need a number of computations comparing orders of various finite groups; these can be found in the appendix.

\subsection*{Acknowledgements.} The authors are grateful to Alastair Litterick,
Yuri Santos Rego and Stefan Witzel for many helpful conversations. The second- and third-named authors were supported by the SFB 701 of Bielefeld University.

\section{Preliminaries}

\subsection{Notation}
We recall a few conventions which we use frequently. When it is clear, the prime number $p$ will denote the cyclic group of that order. Furthermore, the elementary abelian group of order $p^n$ will be denoted as $p^n$. These conventions are standard in finite group theory.
We also follow Artin's convention, and use $\L_n(q)$ to denote $\PSL_n$ over the field of cardinality $q$.

We conjugate on the right, and use the following commutator convention
\[
 [g,h] = ghg^{-1}h^{-1}
\]

The abstract symmetric group of degree $n$ is denoted by $S_n$. Given a set $I$, we define $\Sym(I)$ to be its symmetric group.
We define $A_n$ and $\Alt(I)$ in the analogous manner for the alternating groups.

We fix $n$ and denote by $N$ the set $\{1, \dots, n\}$.

\subsection{Some subgroups and elements of \texorpdfstring{$\Aut(F_n)$}{Aut(Fn)}}

We pick a free generating set $a_1, \dots, a_n$ for the free group $F_n$. 
We will abuse notation by writing $F_n = F(N)$, and given a subset $I \subseteq N$ we will write $F(I)$ for the subgroup of $F(N)$ generated by the elements $a_i $ with $i \in I$.

For every $i,j \in N$, $i\neq j$, set
\begin{align*}
\rho_{ij}(a_k) &= \left\{ \begin{array}{lcl}
                            a_i a_j & \textrm{ if } & k = i \\
                            a_k\phantom{{a_i}^{-1}} & \textrm{ otherwise}
                           \end{array} \right. \\
\lambda_{ij}(a_k) &= \left\{ \begin{array}{lcl}
                            a_j a_i & \textrm{ if } & k = i \\
                            a_k\phantom{{a_i}^{-1}} & \textrm{ otherwise}
                           \end{array} \right. \\
\sigma_{ij}(a_k) &= \left\{ \begin{array}{lcl}
                            a_j\phantom{{a_i}^{-1}} & \phantom{\textrm{ otherwise}} & k = i \\
                            a_i & \textrm{ if } & k=j \\
                            a_k & \phantom{\textrm{ otherwise}} & k \not\in \{i,j\}
                           \end{array} \right.
                           \end{align*}
\begin{align*}
\sigma_{i (n+1)}(a_k) &= \left\{ \begin{array}{lcl}
                            {a_i}^{-1} & \textrm{ if } & k = i \\
                            a_k{a_i}^{-1} & \textrm{ otherwise}
                           \end{array} \right. \\
\epsilon_{i}(a_k) &= \left\{ \begin{array}{lcl}
                            {a_i}^{-1} & \textrm{ if } & k = i \\
                            a_k\phantom{{a_i}^{-1}} & \textrm{ otherwise}
                           \end{array} \right.
                           \\
\delta(a_k) &=  \hspace{6pt} \left. \begin{array}{lcl}
                            {a_k}^{-1}\phantom{a_i} & \hspace{2pt} \textrm{ for every} \hspace{1pt} & k
                            \end{array} \right.
\end{align*}
All of the endomorphisms of $F_n$ defined above are in fact elements of $\Aut(F_n)$. The elements $\rho_{ij}$ are the \emph{right transvections}, the elements $\lambda_{ij}$ are the \emph{left transvections}, and the set of all transvections generates $\SAut(F_n)$.

The involutions $\epsilon_i$ pairwise commute, and hence generate $2^n$ inside $\Aut(F_n)$. We have $2^{n-1} = 2^n \cap \SAut(F_n)$. When talking about $2^n$ or $2^{n-1}$ inside $\Aut(F_n)$, we will always mean these subgroups.

The elements $\sigma_{ij}$ with $i,j \in N$ generate a symmetric group $S_n$. Each of the sets
\[
 \{\epsilon_i \mid i \in N\}, \{\rho_{ij} \mid i,j \in N\} \textrm{ and } \{\lambda_{ij} \mid i,j \in N\}
\]
is preserved under conjugation by elements of $S_n$, and the left conjugation coincides with the natural action by $S_n$ on the indices.
We have \[A_n = S_n \cap \SAut(F_n)\]

The elements $\sigma_{ij}$ with $i,j \in N \cup \{n+1\}$  generate a symmetric group $S_{n+1}$. Again, we have $A_{n+1} = S_{n+1} \cap \SAut(F_n)$.
Again, when we talk about $A_n, S_n, A_{n+1}$ or $S_{n+1}$ inside $\Aut(F_n)$, we mean these subgroups.

Since the symmetric group $S_n$ acts on $2^n$ by permuting the indices of the elements $\epsilon_i$, we have $2^n \rtimes S_n < \Aut(F_n)$ (note that this is the Coxeter group of type $\B_{n}$). As usual, we will refer to this specific subgroup as $2^n \rtimes S_n$.
Clearly, $\SAut(F_n) \cap 2^n \rtimes S_n$ contains $2^{n-1} \rtimes A_n$. 
We will denote this subgroup as $D'_n$ (since it is isomorphic to the derived subgroup of the Coxeter group of type $\D_n$). Note that $2^{n-1}$ inside $D'_n$ is generated by the elements $\epsilon_i \epsilon_j$ with $i \neq j$.

\begin{lem}
\label{closure of An in Dn}
Let $n\geqslant 3$. Then the normal closure of $A_n$ in $D'_n$ is the whole of $D'_n$.
\end{lem}
\begin{proof}
 \[
  \epsilon_1 \epsilon_2 = [\epsilon_1 \epsilon_3, \sigma_{12}\sigma_{13}] \in \langle \! \langle A_n \rangle \! \rangle
 \]
and so every $\epsilon_i \epsilon_j$ lies in the normal closure of $A_n$ as well, since $A_n$ acts transitively on unordered pairs in $N$.
\end{proof}

The Nielsen Realisation theorem for free groups (proved independently by Culler \cite{Culler1984}, Khramtsov~\cite{Khramtsov1985} and Zimmermann~\cite{Zimmermann1981}) states that every finite subgroup of $\Aut(F_n)$ can be seen as a group of basepoint preserving automorphisms of a graph with fundamental group identified with $F_n$. From this point of view, the subgroup $2^n \rtimes S_n$ is the automorphism group of the $n$-rose (the bouquet of $n$ circles), and the subgroup $S_{n+1}$ is the basepoint preserving automorphism group of the  $(n+1)$-cage graph (a graph with two vertices and $n+1$ edges connecting them).

\begin{rmk}
 Throughout, we are going to make extensive use of the Steinberg commutator relations in $\Aut(F_n)$, that is
 \[
  {\rho_{ij}}^{-1} = [{\rho_{ik}}^{-1},{\rho_{kj}}^{-1}]
 \]
and
 \[
  {\lambda_{ij}}^{-1} = [{\lambda_{ik}}^{-1},{\lambda_{kj}}^{-1}]
 \]
 We will also use
 \[
 {\rho_{ij}}^{\delta} = \lambda_{ij}
 \]
\end{rmk}
Another two types of relations which we will frequently encounter are already present in the proof of the following lemma (based on observations of Bridson--Vogtmann~\cite{BridsonVogtmann2003}).
\begin{lem}
\label{conjugate transvections}
 For $n \geqslant 3$, all automorphisms ${\rho_{ij}}^{\pm 1}$ and ${\lambda_{ij}}^{\pm 1}$ (with $i \neq j$)
are conjugate inside $\SAut(F_n)$.
\end{lem}
\begin{proof}
 Observe that
 \[ {\rho_{ij}}^{\epsilon_i \epsilon_j} = \lambda_{ij}\]
 and that
 \[ {\rho_{ij}}^{\epsilon_j \epsilon_k} = {\rho_{ij}}^{-1}\]
 where $k \not\in \{i,j\}$.

 When $n \geqslant 4$, the subgroup $A_n$ acts transitively on ordered pairs in $N$, and so we are done.

 Let us suppose that $n=3$. In this case, using the 3-cycle $\sigma_{12} \sigma_{23}$, we immediately see that $\rho_{12}, \rho_{23}$ and $\rho_{31}$ are all conjugate. We also have
 \[
  {\rho_{12}}^{\sigma_{12} \epsilon_3} = \rho_{21}
 \]
and thus all right transvections are conjugate, and we are done.
\end{proof}

We also note the following useful fact, following from Gersten's presentation of $\SAut(F_n)$~\cite{Gersten1984}.
\begin{prop}
\label{perfect}
 For every $n \geqslant 3$, the group $\SAut(F_n)$ is perfect.
\end{prop}

\subsection{Linear quotients}

Observe that abelianising the free group $F_n$ gives us a map
\[
\SAut(F_n) \to \SL_n(\Z)
\]
This homomorphism is in fact surjective, since each elementary matrix in $\SL_n(\Z)$ has a transvection in its preimage. We will refer to this map as the natural homomorphism $\SAut(F_n) \to \SL_n(\Z)$.

The finite quotients of $\SL_n(\Z)$ (for $n\geqslant 3$) are controlled by the Congruence Subgroup Property as proven by Mennicke~\cite{Mennicke1965}). In particular, noting that $\SAut(F_n)$ is perfect (\cref{perfect}), we conclude that the non-trivial simple quotients of $\SL_n(\Z)$ are the groups $\L_n(p)$ where $p$ ranges over all primes. The smallest one is clearly $\L_n(2)$.

We will refer to the compositions of the natural map  $\SAut(F_n) \to \SL_n(\Z)$ and the quotient maps $ \SL_n(\Z) \to \L_n(p)$ as natural maps as well.

We will find the following observations (due to Bridson--Vogtmann~\cite{BridsonVogtmann2003}) most useful.

\begin{lem}
\label{killing delta}
\label{killing epsilon}
\label{killing Dn}
 Let $n \geqslant 3$, and let $\phi$ be a homomorphism with domain $\SAut(F_n)$.
 \begin{enumerate}
  \item If $n$ is even, and $\phi(\delta)$ is central in $\im \phi$, then $\phi$ factors through the natural map $\SAut(F_n) \to \SL_n(\Z)$.
  \item For any $n$, if there exists $\xi \in 2^{n-1} \s- \{1,\delta\}$ such that $\phi(\xi)$ is central in $\im \phi$, then $\phi$ factors through the natural map $\SAut(F_n) \to \L_n(2)$.
  \item For any $n$, if there exists $\xi \in D_n' \s- 2^{n-1}$ such that $\phi(\xi)$ is central in $\im \phi$, then $\phi$ is trivial.
 \end{enumerate}
\end{lem}
\begin{proof}
\noindent {(1)}
We have
\[
 \delta \rho_{ij} \delta = \lambda_{ij}
\]
for every $i,j$. Thus, $\phi$ factors through the group obtained by augmenting Gersten's presentation~\cite{Gersten1984} of $\SAut(F_n)$ by the additional relations $\rho_{ij} = \lambda_{ij}$. But this is equivalent to Steinberg's presentation of $\SL_n(\Z)$.

\medskip
\noindent {(2)}
We claim that there exists $\tau \in A_n$ such that
\[
 [\xi , \tau] = \epsilon_i \epsilon_j
\]
We have $\xi = \prod_{i \in I} \epsilon_i$ for some $I \subset N$. Take $i \in I$ and $j \not\in I$. Suppose first that there exist distinct $\alpha, \beta$ either in $I \s- \{i\}$ or in $N \s- (I \cup \{j\})$. Then $\tau = \sigma_{ij} \sigma_{\alpha \beta}$ is as claimed.

If no such $\alpha$ and $\beta$ exist, then $n\leqslant 4$ and $\xi = \epsilon_i \epsilon_{i'}$ for some $i' \neq i$. Thus we may take $\tau = \sigma_{i i'} \sigma_{j i}$. This proves the claim.

\smallskip
Now $\phi([\xi , \tau])=1$ since $\phi(\xi)$ is central.
Using the action of $A_n$ on $2^{n-1}$ we immediately conclude that $2^{n-1} \leqslant \ker \phi$. Thus we have
 \[
  \phi(\rho_{ij}) = \phi(\rho_{ij})^{\phi(\epsilon_i \epsilon_j)} =  \phi({\rho_{ij}}^{\epsilon_i \epsilon_j}) =  \phi(\lambda_{ij})
 \]
Now Gersten's presentation tells us that $\phi$ factors through the natural map \[\SAut(F_n) \to \SL_n(\Z)\]
Moreover,
 \[
  \phi(\rho_{ij}) = \phi(\rho_{ij})^{\phi(\epsilon_j \epsilon_k)} =  \phi({\rho_{ij}}^{\epsilon_j \epsilon_k}) =  \phi({\rho_{ij}}^{-1})
 \]
where $k \not\in\{i,j\}$. The result of Mennicke~\cite{Mennicke1965} tells us that in this case $\phi$ factors further through $\SL_n(\Z) \to \L_n(2)$.

\medskip
\noindent {(3)}
We write $\xi = \xi' \tau$, where $\xi' \in 2^{n-1}$ and $\tau \in A_n$.

Suppose first that $\tau$ is not a product of commuting transpositions. Then, without loss of generality, we have
\[
 {\rho_{12}}^\xi = {x_{23}}^{\pm 1}
\]
where $x \in \{ \rho , \lambda\}$. Now
\[
 \phi({\rho_{13}}^{-1}) = \phi([{\rho_{12}}^{-1},{\rho_{23}}^{-1}])= [\phi({x_{23}})^{\mp 1},\phi({\rho_{23}})^{-1}] = 1
\]
as ${x_{23}}^{\pm 1}$ commutes with ${\rho_{23}}$. This trivialises $\phi$, since $\SAut(F_n)$ is generated by transvections, and every two transvections are conjugate.

Now suppose that $\tau$ is a product of commuting transpositions. Then $n \geqslant 4$, and without loss of generality
\[
 {\rho_{12}}^\xi = {x_{34}}^{\pm 1}
\]
where $x \in \{ \rho , \lambda\}$. Now
\[
 \phi({\rho_{14}}^{-1}) = \phi([{\rho_{12}}^{-1},{\rho_{24}}^{-1}])= [\phi({x_{34}})^{\mp 1},\phi({\rho_{24}})^{-1}] = 1
\]
as ${x_{34}}^{\pm 1}$ commutes with ${\rho_{24}}$. This trivialises $\phi$ as before.
\end{proof}

\begin{rmk}
\label{killing epsilon rem}
We can draw the same conclusion as in (2) if
\[\phi(\rho_{ij}) = \phi(\lambda_{ij}) = \phi(\rho_{ij})^{-1}\]
\end{rmk}

\begin{cor}
\label{smallest linear quotient}
Let $\phi \colon \SAut(F_n) \to K$ be a homomorphism. If $K$ is finite and $\phi\vert_{D_n'}$ is not injective, then \begin{enumerate}
\item $\phi$ is trivial; or
\item $\vert K \vert > \vert \L_n(2) \vert$; or
\item $K \cong \L_n(2)$ and $\phi$ is the natural map up to postcomposition with an automorphism of $\L_n(2)$.
\end{enumerate}\end{cor}

Many parts of the current paper are inductive in nature, and they are all based on the following observation.
\begin{lem}
 \label{centralisers}
 For any $k\leqslant n+1$, the group $\SAut(F_n)$ contains an element $\xi$ of order $k$ whose centraliser contains $\SAut(F_{n-k})$. When $k$ is odd then the centraliser contains $\SAut(F_{n-k+1})$. Moreover, when $k \geqslant 5$ the element $\xi$ can be chosen in such a way that the normal closure of $\xi$ is the whole of $\SAut(F_n)$.
\end{lem}
\begin{proof}
Let $\Gamma$ be a graph with vertex set $\{u, v\}$ and edge set consisting of $k$ distinct edges connecting $u$ to $v$ and $n-k+1$ distinct edges running from $v$ to itself (see \cref{pic1}). Let $K$ denote the set of the former $k$ edges. We identify the fundamental group of $\Gamma$ with $F_n$. Elements of $\SAut(F_n)$ then correspond to based homotopy equivalences of $\Gamma$.

Let $\xi$ denote the automorphism of $\Gamma$ which cyclically permutes the edges of $K$. When $k$ is odd the action of $\xi$ on the remaining edges is trivial; when $k$ is even, then $\xi$ acts on $n-k$ of the remaining edges trivially, but it flips the last edge. It is clear that $\xi$ defines an element of $\SAut(F_n)$ of order $k$. Also, $\xi$ fixes pointwise a free factor of $F_n$ of rank $n-k$ when $k$ is even or $n-k+1$ when $k$ is odd.

Observe that we have a copy of the alternating group $A_k$ permuting the edges in $K$.
Suppose that $k\geqslant 5$, and let $\tau$ denote some $3$-cycle in $A_k$. It is obvious that $[\xi,\tau]$ is a non-trivial element of $A_k$; it is also clear that $A_k$ is simple. These two facts imply that $A_k$ lies in the normal closure of $\xi$. But this  $A_k$ contains an $A_{k-1}$ contained in the standard $A_n < \SAut(F_n)$, and the result follows by \cref{killing Dn}(3).
\end{proof}

\begin{figure}[h]
 \includegraphics{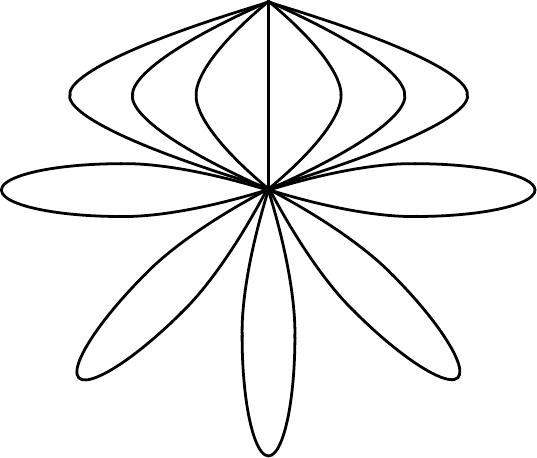}
 \caption{The graph $\Gamma$ with $(n,k) = (11,7)$}
 \label{pic1}
\end{figure}

One can easily give an algebraic description of the element $\xi$ above; in fact we will do this for $k=3$ when we deal with classical groups in characteristic 3 in \cref{sec: char 3}.

\section{Alternating groups}

In this section we will give lower bounds on the cardinality of a set on which the groups $\SAut(F_n)$ can act non-trivially.
The cases $n \in \{3, \dots, 8\}$ are done in a somewhat ad-hoc manner, and we begin with these, developing the necessary tools along the way. We will conclude the section with a general result for $n \geqslant 9$.

As a corollary, we obtain that alternating groups are never the smallest quotients of $\SAut(F_n)$ (for $n \geqslant 3$), with a curious exception for $n=4$, since in this case the (a fortiori smallest) quotient $\L_4(2)$ is isomorphic to the alternating group $A_8$.

\begin{lem}[$n=3$]
\label{A_5 not quotient of F_3}
 Any action of $\SAut(F_3)$ on a set $X$ with fewer than $7$ elements is trivial.
\end{lem}
This result can be easily verified using GAP. For this reason, we offer only a sketch proof.
\begin{proof}[Sketch of proof]
The action gives us a homomorphism $\phi \colon \SAut(F_3) \to S_6$. Since $\SAut(F_3)$ is perfect, the image lies in $A_6$.

Consider the set of transvections
\[
 T = \{ {\rho_{ij}}^{\pm 1}, {\lambda_{ij}}^{\pm 1} \}
\]
By \cref{conjugate transvections} all elements in $T$ are conjugate in $\SAut(F_3)$, and hence also in the image of $\phi$.

Consider the equivalence relation on $T$ where elements $x,y \in T$ are equivalent if $\phi(x)=\phi(y)$.
It is clear that the equivalence classes are equal in cardinality.
We first show that if any (hence each) of these equivalence classes has cardinality at least $3$, then $\phi$ is trivial.
Let $\Phi$ be such a class. Without loss of generality we assume that $\rho_{12} \in \Phi$.

Suppose that
\[\vert \Phi \cap \{ {\rho_{12}}^{\pm1}, {\lambda_{12}}^{\pm1} \} \vert \geqslant 3 \]
Independently of which $3$ of the $4$ elements lie in $\Phi$, their image under $\phi$ is centralised by $\phi(\epsilon_1 \epsilon_2)$, and so in fact all four of these elements lie in $\Phi$.

Now, we have
\[
 \phi(\rho_{12}) = \phi({\rho_{12}}^{-1}) = \phi(\lambda_{12}) = \phi({\lambda_{12}}^{-1})
\]
In this case we conclude from \cref{killing epsilon rem} that $\phi$ factors through $\L_3(2)$. But $\L_3(2)$ is a
simple group containing an element of order 7, and so
$\phi$ is trivial.

If $\vert \Phi \cap \{ {\rho_{12}}^{\pm1}, {\lambda_{12}}^{\pm1} \} \vert < 3$ then
there exists an element ${x_{ij}}^{\pm1} \in \Phi$ with $x$ being either $\rho$ or $\lambda$, and with $(i,j) \neq (1,2)$.
If $(i,j) = (1,3)$ then
\[
 \phi({\rho_{13}}^{-1}) = \phi([{\rho_{12}}^{-1},{\rho_{23}}^{-1}]) = [\phi(x_{13})^{\mp 1},\phi(\rho_{23})^{-1}] = 1
\]
which implies that $\phi$ is trivial. We proceed in a similar fashion for all other values of $(i,j)$.
This way we verify our claim that if $\phi$ is non-trivial then $\vert \Phi \vert \leqslant 2$.

\smallskip

There are exactly 10 elements in $T$ which commute with $\rho_{12}$, namely
\[
 C_T(\rho_{12}) = \{{\rho_{12}}^{\pm1}, {\lambda_{12}}^{\pm1}, {\lambda_{13}}^{\pm1},  {\rho_{32}}^{\pm1}, {\lambda_{32}}^{\pm1}\}
\]
Using what we have learned above about the cardinality of $\Phi$, we see that $\phi$ is not trivial only if
the conjugacy class of $\phi(\rho_{12})$ in $A_6$ contains at least 5 elements commuting with $\phi(\rho_{12})$. We also know that this conjugacy class has to contain $\phi(\rho_{12})^{-1}$. By inspection we see that the only such conjugacy class in $A_6$ is that of $\tau = (12)(34)$.

The conjugacy class of $\tau$ has exactly 5 elements, say $\{\tau, \tau_1, \tau_2, \tau_1', \tau_2' \}$, and the elements $\tau_i$ and $\tau_j'$ do not commute for any $i,j \in \{ 1,2\}$. Hence the maximal subset of the conjugacy class of $\tau$ in which all elements pairwise commute is of cardinality 3. But in $C_T(\rho_{12})$ we have 8 such elements, namely
\[
 \{{\rho_{12}}^{\pm1}, {\lambda_{12}}^{\pm1}, {\rho_{32}}^{\pm1}, {\lambda_{32}}^{\pm1}\}
\]
This implies that $\vert \Phi \vert >2$, which forces $\phi$ to be trivial.
\end{proof}

Note that $\L_3(2)$ acts non-trivially on the set of non-zero vectors in $2^3$ (thought of as a vector space) which has cardinality $7$, and so $\SAut(F_3)$ has a non-trivial action on a set of $7$ elements. Thus the result above is sharp.

\begin{lem}
 \label{way too big}
 Let $n \geqslant 4$.
Suppose that $\SAut(F_n)$ acts non-trivially on a set $X$ so that $\epsilon_1 \epsilon_2$ acts trivially. Then $\vert X \vert \geqslant 2^{n-1}$.
\end{lem}
\begin{proof}
 When $\epsilon_1 \epsilon_2$ acts trivially then the action factors through the natural map $\SAut(F_n) \to \L_n(2)$ by \cref{killing epsilon}. Now, $\L_n(2)$ is simple and cannot act on a set smaller than $2^{n-1}$ non-trivially, by \cite[Theorem 5.2.2]{KleidmanLiebeck1990}.
\end{proof}


\begin{lem}[$n=4$]
\label{actions n=4}
 Any action of $\SAut(F_4)$ on a set $X$ with fewer than $8$ elements is trivial.
\end{lem}
\begin{proof}
The group $D'_4$ cannot act faithfully on fewer than $8$ points, and the group $D_4'/\langle \delta \rangle$ cannot act faithfully on fewer than $12$ points (both of these facts can be easily checked). Thus we are done by \cref{killing delta}, since $\L_4(2) \cong A_8$ cannot act on fewer than $8$ points.
\end{proof}

Note that in this case the result is also sharp, since we have the epimorphism \[\SAut(F_4) \to \L_4(2) \cong A_8 \]

\begin{lem}\label{fixed by A_n}
Let $n \geqslant 1$.
 Let $D'_n$ act on a set of cardinality less than $2^{n-1-p(n)}$ where $p(n)=0$ for $n$ odd, and $p(n)=1$ for $n$ even. Then $2^{n-1}$ acts trivially on every point fixed by $A_n$.
\end{lem}
\begin{proof}
Let $x$ be a point fixed by $A_n$, and consider the $2^{n-1}$-orbit thereof. Such an orbit corresponds to a subgroup of $2^{n-1}$ normalised by $A_n$. It cannot correspond to the trivial subgroup, as then the orbit would be too large. For the same reason it cannot correspond to the subgroup generated by $\delta$ (which is a subgroup when $n$ is even). It is easy to see that the only remaining subgroup is the whole of $2^{n-1}$, and so the action on $x$ is trivial.
\end{proof}
\begin{lem}
\label{actions single orbit}
Let $n\geqslant 4$.
Suppose that $\SAut(F_n)$ acts on a set $X$ of cardinality less than $2^{n-1-p(n)}$, where $p(n)$ is as above, in such a way that $A_n$ has at most  one non-trivial orbit. Then $\SAut(F_n)$ acts trivially on $X$.
\end{lem}
\begin{proof}
By \cref{fixed by A_n}, every point fixed by $A_n$ is also fixed by $2^{n-1}$. This implies that every $A_n$-orbit is already a $D'_n$-orbit, and so every point in $X$ is  fixed by $2^{n-1}$, and thus the action of $\SAut(F_n)$ on $X$ is trivial by \cref{way too big}.
\end{proof}

For higher values of $n$ we will use the fact that actions of alternating groups on small sets are well-understood.
Recall that $A_n = \Alt(N)$ denotes the  group of even permutations of the set $N = \{1,\dots,n\}$.

\begin{dfn}
We say that a transitive action of $A_n = \Alt(N)$ on a set $X$ is \emph{associated to $k$} (with $k \in \N$) if for every (or equivalently any) $x \in X$ the stabiliser of $x$ in $A_n$ contains $\Alt(J)$ with
$J \subseteq N$ and $\vert J \vert = k$, and does not contain $\Alt(J')$ for any $J'\subseteq N$ of larger cardinality. Notice that $k \geqslant 2$.
\end{dfn}

In the following lemma we write $A_{n-1}$ for the subgroup $\Alt(N\s-\{n\})$ of $A_n = \Alt(N)$.

\begin{lem}
 \label{associated}
 Let $n\geqslant 6$.
Suppose that we are given a transitive action $\pi$ of $A_n$ on a set $X$ which is associated to $k$. Then
\begin{enumerate}
 \item the action of $A_{n-1}$ on each of its orbits is associated to $k$ or $k-1$;
 \item if $k > \frac n 2$ then there is exactly one orbit of $A_{n-1}$ of the latter kind;
  \item if $k > \frac n 2 $ then any other transitive action $\pi'$ of $A_n$ on a set $X'$ isomorphic to $\pi\vert_{A_{n-1}}$ when restricted to $A_{n-1}$ is isomorphic to $\pi$.
\end{enumerate}
\end{lem}
\begin{proof}
Before we start, let us make an observation: let $I$ and $J$ be two subsets of $N$ of cardinality at least $3$ each, and such that $I \cap J \neq \emptyset$. Then
\[\langle \Alt(I), \Alt(J) \rangle = \Alt(I\cup J)\]
There are at least two quick ways of seeing it: the subgroup $\langle \Alt(I), \Alt(J) \rangle$ clearly contains all $3$-cycles; the subgroup $\langle \Alt(I), \Alt(J) \rangle$ acts $2$-transitively, and so primitively, on $I \cup J$, and contains a $3$-cycle, which allows us to use Jordan's theorem.

\medskip
\noindent (1)
Let $x$ be an element in an $A_{n-1}$-orbit $O$. If the stabiliser $S$ of $x$ in $A_n$ contains $\Alt(J)$ with $|J| = k$ and $n \not\in J$, then $\Alt(J) \subseteq A_{n-1}$ and  so the action of $A_{n-1}$ on $O$ is associated to at least $k$.
It is clear that the action cannot be associated to any integer greater than $k$.

If $n$ is in $J$ for every $J \subseteq N$ of size $k$ with $\Alt(J) \subseteq S$, then $\Alt(J\s-\{n\})$ is contained in $S \cap A_{n-1}$ and the action of $A_{n-1}$ on $O$ is associated with at least $k-1$.
It is clear that this action cannot be associated to any integer greater than $k-1$.

\medskip
\noindent (2)
Clearly, there exists $x \in X$ such that its stabiliser $S$ in $A_n$ contains $\Alt(J)$ with $\vert J \vert = k$ and $n\in  J$.
Note that $J$ is unique -- if there were another subset $I \subseteq N$ with $\vert I \vert = k$ and $\Alt(I)\leqslant S$, then $I\cap J$ would need to intersect non-trivially (as $k> \frac n 2$), and so we would have $\Alt(I \cup J) \leqslant S$.
Hence we may conclude from the proof of (1) that the action of $A_{n-1}$ on $O$, its orbit of $x$, is associated to $k-1$.

Now suppose that there exists a point $x' \in X$ with $A_{n-1}$-orbit $O'$, stabiliser $S'$ in $A_n$, and subset $J' \subseteq N$ of cardinality $k$ with $n \in J'$ and $\Alt(J')\leqslant S'$. There exists $\tau \in A_n$ such that
\[
x = \tau.x'
\]
and so $S' = S^\tau$. Thus $\tau(J') = J$, and therefore there exists $\sigma \in \Alt(J)$ such that $\sigma \tau(n) = n$. But then also $\sigma^{-1}.x = x$ and so $x = \sigma \tau.x'$. But $\sigma \tau \in A_{n-1}$, and therefore $O = O'$.

\medskip
\noindent (3)
Let us start by looking at $\pi'$. By assumption, this action is associated to at least $k-1$, since it is when restricted to $A_{n-1}$. It also cannot be associated to any integer larger than $k$, since then (2) would forbid the existence of an $A_{n-1}$-orbit in $X'$ associated to $k-1$, and we know that such an orbit exists.

If $k-1 > \frac n 2 $ then (2) implies that
$\pi'\vert_{A_{n-1}}$ has an orbit that is associated to $k-1$, but clearly none associated to $k-2$. Thus, by (2),  $\pi'$ is associated to $k$.

If $k-1\leqslant \frac n 2$ then
in particular $k \neq n$, and so there is an $A_{n-1}$-orbit in $X$ associated to $k$, and therefore $\pi'$ cannot be associated to $k-1$. We conclude that $\pi'$ is associated to $k$.

Pick an $x \in X$ so that
the $A_{n-1}$-action on the $A_{n-1}$ orbit $O$ of $x$ is associated to $k-1$. Let $\theta \colon X \to X'$ be a $A_{n-1}$-equivariant bijection (which exists by assumption). Let $x' = \theta(x)$. Let $S$ denote the stabiliser of $x$ in $A_n$, and $S'$ the stabiliser of $x'$ in $A_n$.

We have $\Alt(J) \times G = H \leqslant S$, where $\vert J \vert = k$, the index $\vert S:H\vert$ is at most $2$, and $G \leqslant \Alt(N \s- J)$ -- we need to observe that $J$ is unique, as proven above.
Since the $A_{n-1}$-orbit of $x$ is associated to $k-1$, we see that $n \in J$.

Now the stabiliser of $x'$ in $A_{n-1}$ is equal to $S \cap A_{n-1}$, and $S'$ contains $S \cap A_{n-1}$ and some $\Alt(J')$ with $\vert J' \vert = k$.
Since $k > \frac n 2$, the subsets $J$ and $J'$ intersect, and so $\Alt(J\cup J') \leqslant S'$ as before. But the $A_{n-1}$-action on its orbit of $x'$ is associated to $k-1$, and so we must have $J = J'$.
This implies that $S'=S$, since the index of $H$ in $S$ is equal to the index of $H \cap A_{n-1}$ in $S \cap A_{n-1}$.
\end{proof}

\begin{thm}[{Dixon--Mortimer~\cite[Theorem 5.2A]{DixonMortimer1996}}]
\label{dixonmortimer}
 Let $n \geqslant 5$, and let $r \leqslant n/2$ be an integer. Let $H < A_n = \Alt(N)$ be a proper subgroup of index less than $\binom n r$. Then one of the following holds:
 \begin{enumerate}
  \item The subgroup $H$ contains a subgroup $A_{n-r+1}$ of $A_n$ fixing $r-1$ points in $N$.
  \item We have $n=2m$ and $\vert A_n : H \vert = \frac 1 2 \binom n m$. Moreover, $H$ contains the product $A_m \times A_m$.
  \item The pair $(n,\vert A_n : H \vert)$ is one of the six exceptional cases:
  \[
   (5,6), (6,6), (6,15), (7,15), (8,15), (9,120)
  \]
 \end{enumerate}
\end{thm}
Note that the original theorem contains more information in each of the cases; for our purposes however, the above version will suffice.

We can rephrase (1) by saying that the action $A_n \curvearrowright A_n/H$ is associated to at least $n-r+1$.

\begin{cor}[$n=5$]
\label{actions on set n=5}
Every action of $\SAut(F_5)$ on a set $X$ of cardinality less than $12$ is trivial.
\end{cor}
\begin{proof}
By \cref{dixonmortimer}, the only orbits of $A_6$ in $X$ are of cardinality $1$, $6$ or $10$. There can be at most one orbit of size greater than 1, and it contains at most two non-trivial orbits of $A_5$, since such orbits have cardinality at least 5. If there is at most one non-trivial $A_5$-orbit, we invoke \cref{actions single orbit}. Otherwise, let $x$ be a point on which $A_6$ acts non-trivially; its $A_6$-orbit consists of $10$ points. We know from case (2) of \cref{dixonmortimer} that it is fixed by two commuting 3-cycles. The $A_5$-orbit of $x$ has cardinality 5, and so it is the natural $A_5$-orbit (by \cref{dixonmortimer} again). Thus $x$ is fixed by some standard $A_4$. But now any standard $A_4$ together with any two commuting 3-cycles generates $A_6$, and so $x$ is fixed by $A_6$, which is a contradiction.
\end{proof}

\begin{cor}[$n=6$]
\label{actions on set n=6}
Every action of $\SAut(F_6)$ on a set $X$ of cardinality less than $14$ is trivial.
\end{cor}
\begin{proof}
By \cref{dixonmortimer}, the only orbits of $A_7$ in $X$ are either trivial or the natural orbits of size 7. There can be at most one such natural orbit, and so $A_6$ has at most one non-trivial orbit. We now invoke \cref{actions single orbit}.
\end{proof}

We now begin the preparations towards the main tool in this section.

\begin{lem}
\label{stabilisers}
Let $n \geqslant 2$, and let $r \leqslant n/2$ be a positive integer.
 Let $D'_n$ act on a set $X$ of cardinality less than $\binom n r$, and let $x$ be a point
stabilised by
\[
  \langle \{ \epsilon_i \epsilon_j \mid i,j \in I \} \rangle
\]
where $I$ is a subset of $N$.
Then $I$ can be taken to have cardinality at least $n-r + 1$, provided that
\begin{enumerate}
\item $I$ contains more than half of the points of $N$; or
\item $I$ contains exactly half of the points, and $x$ is not fixed by some $\epsilon_i \epsilon_j$ with $i,j \not\in I$.
\end{enumerate}
\end{lem}
\begin{proof}
Let $S$ denote the stabiliser of $x$ in $2^{n-1}$.
Consider a maximal (with respect to inclusion) subset $J$ of $N$ such that for all $i,j \in J$ we have $\epsilon_i \epsilon_j \in S$. We call such a subset a \emph{block}.
It is immediate that blocks are pairwise disjoint, and one of them, say $J_0$, contains $I$.

If the block $J_0$ contains more than half of the points in $N$ (which is guaranteed to happen in the case of assumption (1)), then it is the unique largest block of $S$.
In the case of assumption (2), the block $J_0$ may contain exactly half of the points, but all of the other blocks contain strictly fewer elements. Thus, again, $J_0$ is the unique largest block.

Since $J_0$ is unique, it is clear that any element $\tau \in A_n$ which does not preserve $J_0$ gives $S^\tau \neq S$, and so in particular $\tau.x \neq x$. Using this argument we see that $X$ has to contain at least $\binom n {n-\vert J_0\vert}$ elements.
But $\vert X \vert <  \binom n r$, and so  $\vert J_0 \vert > n-r$, and we are done.
%
\end{proof}

\begin{dfn}
 Let $\SAut(F_n)$ act on a set $X$. For each point $x \in X$ we define
\begin{enumerate}
 \item $I_x$ to be a subset of $N$ such that
  \[
  \langle \{ \epsilon_i \epsilon_j \mid i,j \in I_x \} \rangle
 \]
 fixes $x$, and $I_x$ has maximal cardinality among such subsets.
 \item $J_x$ to be a subset of $N$ such that $x$ is fixed by
 \[\Alt(J_x) \leqslant \Alt(N) = A_n\]
 and $J_x$ is of maximal cardinality among such subsets.
\end{enumerate}
\end{dfn}

The following is the main technical tool of this part of the paper.

\begin{lem}
\label{actions main trick}
 Let $n\geqslant 5$. Suppose that $\SAut(F_n)$ acts transitively on a set $X$ in such a way that
 \begin{enumerate}
  \item there exists a point $x_0 \in X$ with $I_{x_0}$ containing more than half of the points in $N$; and
\item for every $x \in X$ we have $ \vert J_x \vert \geqslant  \frac{n+3} 2$.
 \end{enumerate}
 Then every point $x$ is fixed by $\SAut(F(J_x))$, provided that $\vert X \vert < \min\{2^{n-r}, \binom n r\}$ for some positive integer $r < \frac n 2 -1$.
\end{lem}
\begin{proof}
\cref{stabilisers} tells us immediately that $I_{x_0}$ contains at least $\nu = n-r+1$ points. We claim that in fact every $I_x$ contains at least $\nu$ points.

Since the action of $\SAut(F_n)$ on $X$ is transitive, and $\SAut(F_n)$ is generated by transvections, it is enough to prove that for every point $x$ with $I_x$ of size at least $\nu$, and every transvection, the image $y$ of $x$ under the transvection has $\vert I_y \vert \geqslant \nu$. For concreteness, let us assume that the transvection in question is $\rho_{ij}$ (the situation is analogous for the left transvections).
Since $\rho_{ij}$ commutes with every involution $\epsilon_\alpha \epsilon_\beta$ with $\alpha, \beta \in I_x \s- \{i,j\}$, we see that $y$ is fixed by $\epsilon_\alpha \epsilon_\beta$. But
\[ \vert I_x\vert  - 2 \geqslant \nu-2 = n-r-1 > \frac n 2
 \]
and so
$I_y \supseteq I_x \s- \{i,j\}$ (here we use the fact that $I_y$ is defined to be the largest block).
Therefore $\vert I_y \vert \geqslant \nu$ by \cref{stabilisers}. We have thus established that $I_x$ contains at least $\nu$ points for every $x \in X$.

\smallskip
Note that the sets $I_x$ form a poset under inclusion. Pick an element $z\in X$ so that $I_z$ is minimal in this poset. Let $Z$ denote the subset of $X$ consisting of points $w$ with $I_w = I_z$. Now for every $i,j \in I_z$  and every $w \in Z$ we have
\[
 I_z \s- \{i,j\} \subseteq I_{\rho_{ij}.w}
\]
since $\rho_{ij}$ commutes with involutions $\epsilon_\alpha \epsilon_\beta$ with $\alpha, \beta \in I_z \s- \{i,j\}$ as before.

Assume there exists $k \in I_{\rho_{ij}.w} \s- I_z$. By definition of $I_{\rho_{ij}.w}$ and using the above inclusion, there exists $l \in I_z \s- \{i,j\}$ such that $\epsilon_l \epsilon_k$ fixes $\rho_{ij}.w$. But $\epsilon_l \epsilon_k$ commutes with $\rho_{ij}$, and hence fixes $w$, which forces $k \in I_z$, a contradiction. Thus
\[
I_{\rho_{ij}.w}   \subseteq I_z
\]
Since $I_z$ is minimal, we conclude that $\rho_{ij}.w \in Z$. An analogous argument applies to left transvections, and so $Z$ is preserved by
\[
 \SAut(F(I_z)) = \langle \{ \rho_{ij}, \lambda_{ij} \mid i,j \in I_z \} \rangle \leqslant \SAut(F_n)
\]
But in the action of $\SAut(F(I_z))$ on $Z$ the involutions $\epsilon_\alpha \epsilon_\beta$ with $\alpha, \beta \in I_z$ act trivially. Therefore this action is trivial by \cref{way too big}, since $X$ has fewer than $2^{n-r}$ points and $n - r+1 > \frac n 2 +2 \geqslant 3 \frac 1 2$.
In particular, we have $I_w \subseteq J_w$ for every $w \in Z$.

Every $w \in Z$ is fixed by $\SAut(F(I_w))$, but also by $\Alt(J_w)$ by assumption. Thus, it is fixed by the subgroup of $\SAut(F_n)$ generated by the two subgroups. It is clear that this is $\SAut(F(J_w \cup I_w))$ and so we have finished the proof for points in $Z$. Now we also see that in fact $I_w = J_w$, since the subgroup of $2^{n-1}$ corresponding to $J_w$ lies in $\SAut(F(J_w))$ and hence fixes $w$.

\smallskip
Let $x \in X$ be any point. Since the action of $\SAut(F_n)$ is transitive, there exists a finite sequence $z=x_0, x_1, \dots, x_{m-1}, x_m = x$ such that for every $i$ there exists a transvection $\tau_i$ with $\tau_i.x_i = x_{i+1}$ (we assume as well that the elements of the sequence are pairwise disjoint).
We claim that every $x_i$ is fixed by $\SAut(F(J_{x_i}))$. Let $i$ be the smallest index so that our claim is not true for $x_i$. As usual, for concreteness, let us assume that $\tau_{i-1} = \rho_{\alpha \beta}$. Note that we cannot have both $\alpha$ and $\beta$ in $J_{x_{i-1}}$, since then the action of $\rho_{\alpha \beta}$ on $x_{i-1}$ would be trivial.

Consider the intersection $(J_{x_{i-1}} \cap J_{x_i}) \s- \{\alpha, \beta\}$. By assumption, the intersection $J_{x_{i-1}} \cap J_{x_i}$ contains at least 3 points, and at most one of these points lies in $\{ \alpha, \beta\}$. Thus there exist $\alpha', \beta' \in J_{x_{i-1}} \cap J_{x_i}$ such that $\rho_{\alpha' \beta'}$ commutes with $\rho_{\alpha \beta}$.
The action of $\rho_{\alpha' \beta'}$ on $x_{i-1}$ is trivial, and thus it must also be trivial on $x_i = \rho_{\alpha \beta}.x_{i-1}$. We also know that $\Alt(J_{x_i})$ acts trivially on $x_i$, and so every right transvection with indices in $J_{x_i}$ acts trivially on $x_i$. This implies that $x_i$ is fixed by $\SAut(F(J_{x_i}))$, which contradicts the minimality of $x_i$, and so proves the claim, and therefore the result.
\end{proof}

\begin{prop}[$n\geqslant 7$]
\label{actions on sets}
Let $n \geqslant 7$. Every action of $\SAut(F_n)$ on a set of cardinality less than \[\max_{r \leqslant \frac n 2 -3} \min\Big\{ 2^{n-r-p(n)}, \binom n r \Big\}\]
is trivial, where $p(n)$ equals $0$ when $n$ is odd and $1$ when $n$ is even.
\end{prop}
\begin{proof}
Let $X$ denote the set on which we are acting. Without loss of generality we will assume that $\SAut(F_n)$ acts on $X$ transitively, and that $X$ is non-empty.

Let $R$ denote a value of $r$ for which \[\max_{r \leqslant \frac n 2 -3} \min\Big\{ 2^{n-r-p(n)}, \binom n r \Big\}\] is attained. Note that $R>1$ by \cref{computation actions 2} for $n\geqslant 8$; a direct computation shows that $R=2$ for $n \in \{7,8\}$.

Let us first look at the action of $A_{n+1}$. Since $\vert X \vert < \binom n R < \binom{n+1}R$, \cref{dixonmortimer} tells us that each orbit of $A_{n+1}$ is
\begin{enumerate}
 \item associated to at least $n-R+1$; or
 \item as described in case (2) of the theorem -- this is immediately ruled out, since $X$ would have to be too large by \cref{computation actions} for $n\geqslant 12$, and by direct computation for $n \in \{8,10\}$; or
 \item one of the two exceptional actions $(8,15)$ or $(9,120)$ as in case (3) of the theorem.
\end{enumerate}
For now let us assume that we are in case (1).
Thus $\vert J_x \vert \geqslant n-R$ for each $x \in X$, and so the action of $\SAut(F_n)$ on $X$ satisfies assumption (2) of \cref{actions main trick}. Also, by \cref{associated}, there is at least one point $y \in X$ with $\vert J_y \vert \geqslant n-R+1$.

Let $J_0$ denote a largest (with respect to cardinality) subset of $N$ such that $\Alt(J_0)$ has a fixed point in $X$. Let $x_0$ be such a fixed point.
Note that $J_0$ has at least $n-R+1$ elements, and $\vert X \vert < 2^{n-R-p(n)}$, which implies by \cref{fixed by A_n} that $I_{x_0}$ contains at least $n-R+1$ elements. As $R < \frac n 2$, we conclude that the action of $\SAut(F_n)$ on $X$ satisfies assumption (1) of \cref{actions main trick}.

We are now in position to apply \cref{actions main trick}. We conclude that $x_0$ is fixed by $\SAut(F(J_{x_0}))$.
Let us consider the graph $\Gamma$ from \cref{pic1} with $k=\vert J_{x_0} \vert +1$ and fundamental group isomorphic to $F_n$. We can choose such an isomorphism so that $\Alt(J_{x_0})$ acts on $\Gamma$ by permuting (in a natural way) all but one of the edges which are not loops. But it is clear that we also have a supergroup $G$ of $\Alt(J_{x_0})$, which is isomorphic to an alternating group of rank $\vert J_{x_0} \vert + 1$, and acts by permuting all such edges.
By construction, $G < \SAut(F(J_{x_0}))$ and so $G.x_0 = x_0$. But now consider the action of $G$ on $X$ -- by \cref{associated}, it has to agree with the action of $\Alt(J')$, where $J'$ is a superset of $J_{x_0}$ with a single new element. By assumption, $\Alt(J')$ does not fix any point in $X$. However $G$ does, and this is a contradiction. This implies that there is no superset $J'$, but then we must have $J_{x_0} = N$, and so $\SAut(F_n)$ fixes a point in $X$. But the action is transitive, and so $X$ is a single point. This proves the result.

\smallskip
Now let us investigate the exceptional cases. The first one occurs when $n=7$, and the $A_8$-orbit of $x$ has cardinality 15.
We have 
$R=2$ in this case, and so $X$ has fewer than $\binom 7 2 =  21$ elements.
In this case, there are at most $5$ points in $X \s- A_8.x$, and so the action of $A_8$ on each of these is trivial. Thus $A_7$ also fixes these points.

Now consider the action of $A_7$ on $A_8.x$. Since $A_7$ cannot fix any point here, and the smallest orbit of $A_7$ of size other than 1 and 7 has to be of size 15 by \cref{dixonmortimer}, we conclude, noting that $15=2\cdot7+1$, that $A_7$ acts transitively on $A_8.x$.
Therefore the action of $A_7$ on $X$ has exactly $1$ non-trivial orbit, and so
we may apply \cref{actions single orbit} -- note that $X$ has fewer than $2^6$ points.

\smallskip
The remaining case occurs for $n=8$; we have $R = 2$ and so $X$ has fewer than $\binom 8 2 =  28$ elements. But then we cannot have an orbit of size $120$, and thus this exceptional case does not occur.
\end{proof}

\begin{rmk}
 In particular, we can put $r = \lfloor \frac n 2 \rfloor -3$ in the above result; we see that (asymptotically) $\binom n r$ grows much faster than $2^{n-r}$ (in fact it grows like $n^{-1/2} 2^n$), and so we obtain an exponential bound on the size of a set on which we can act non-trivially. The smallest set with a non-trivial action of $\SAut(F_n)$ known is also exponential in size -- coming from the action of $\L_n(2)$ on the cosets of its largest maximal subgroup (see \cite[Table 5.2.A]{KleidmanLiebeck1990}). Hence the result above answers the question about the asymptotic size of such a set.
\end{rmk}

\begin{thm}
\label{main thm actions}
Let $n \geqslant 3$. Any action of $\SAut(F_n)$ on a set with fewer than $k(n)$ elements is trivial, where
\[
k(n) = \left\{ \begin{array}{ccl}
7 & & n=3 \\
8 & & n=4 \\
12 & \textrm{ if } & n=5 \\
14 &  & n=6 
\end{array} \right.
\]
and $k(n) = \max \limits_{r \leqslant \frac n 2 -3} \min \{ 2^{n-r-p(n)}, \binom n r \}$ for $n\geqslant 7$,
where $p(n)$ equals $0$ when $n$ is odd and $1$ when $n$ is even.
\end{thm}
\begin{proof}
This follows from \cref{A_5 not quotient of F_3,actions n=4,actions on set n=5,actions on set n=6,actions on sets}.
\end{proof}

As commented after the proofs of \cref{A_5 not quotient of F_3,actions n=4}, the bounds are sharp when $n \in  \{3,4\}$.

\begin{table}[h]
\centering\begin{tabular}{l || r}

                                &Order\\
\hline
\hline
$A_5$ 
        &60\\
$\L_3(2)$ 
        &168\\
\hline
$A_6$ 
                        &360\\
$A_7$                                           &2520\\
$\L_4(2) \cong A_8$                                     &21060\\
\hline
$A_9$                                   &181440\\
$A_{10}$                                &1814400\\
$\L_5(2)$                                       &9999360\\
\hline
$A_{11}$                         &19958400\\
$A_{12}$                                 &239500800\\
$A_{13}$                                 &3113510400\\
$\L_6(2)$                                                       &20158709760\\
\hline
\end{tabular}
\caption{Small alternating groups}
\label{table alternating}
\end{table}

\begin{cor}
\label{no alternating quotients}
Let $n \geqslant 3$ and $K$ be a quotient of $\SAut(F_n)$ with $\vert K \vert \leqslant \L_n(2)$. If $K$ is isomorphic to an alternating group, then $n=4$ and $K \cong A_8 \cong L_4(2)$.
\end{cor}
\begin{proof}
The proof consists of two parts. Firstly, \cref{computation alternaing} tells us that
\[\vert A_{\binom n 2} \vert > \vert \L_n(2)\vert \]
 for $n \geqslant 7$. In view of the bounds in \cref{main thm actions}, this proves the result for $n \geqslant 7$ -- for $n \in \{7,8\}$ we have computed above that $r=2$; for larger values of $n$ we have $2^{n-3} > \binom n 2$ by \cref{computation actions 2}.

Secondly, for $3 \leqslant n \leqslant 6$, \cref{table alternating} lists all alternating groups of degree at least 5 smaller or equal (in cardinality) than $\L_6(2)$. The table also lists the groups $\L_n(2)$ in the relevant range. All these groups are listed in increasing order. The result follows from inspecting the table and comparing it to the bounds in \cref{main thm actions}.
\end{proof}

\subsection{An application}

We record here a further application of the bounds established in \cref{main thm actions}.

\begin{thm}
\label{thm phd dawid}
Let $n\geqslant 12$ be an even integer, and let $m\neq n$ satisfy $m < \binom {n+1} 2$. Then every homomorphism
\[
\phi \colon \Out(F_n) \to \Out(F_m)
\]
has image of cardinality at most $2$.
\end{thm}
\begin{proof}
\cite[Theorems 6.8 and 6.10]{Kielak2013} tell us that $\phi$ has a finite image. Every finite subgroup of $\Out(F_m)$ can be realised by a faithful action on a finite connected graph $\Gamma$ with $\delta(\Gamma) \geqslant 3$ and Euler characteristic $1-m$. These two facts immediately imply that $\Gamma$ has fewer than $2m$ vertices. But now
\[ 2 \cdot \binom {n+1} 2 < \binom {n} 2 + 2 \cdot \binom {n+1} 2 +  \binom {n+2} 2 = \binom n 4\]
and
\[
 2 \cdot \binom {n+1} 2 < 2^{n-5}
\]
for $n\geqslant 14$ by an argument analogous to \cref{computation actions 2}.
Therefore, for $n\geqslant 14$ we have
\[
 2 \cdot \binom {n+1} 2 < \min \Big\{ \binom n r, 2^{n-r-1} \Big\}
\]
with $r=4$ (and such an $r$ satisfies $r \leqslant \frac n 2 -3$).

For $n=12$ we take $r=3$ and compute directly that the inequality also holds.

In any case, the action of $\SOut(F_n)$ (via $\phi$) on the vertices of $\Gamma$ is trivial. Now each vertex has at most $2m-2$ edges emanating from it, and so again the action of $\SOut(F_n)$ on these edges is trivial. Thus $\phi(\SOut(F_n))$ is trivial, which proves the result.
\end{proof}

\section{Sporadic groups}\label{sectionsporadic}

In this section we show that sporadic groups are never the smallest quotients of $\SAut(F_n)$. The proof relies on determining for each sporadic group its \emph{\drank{}}, that is the largest $n$ such that the group contains $D'_n$. This information can be extracted from the lists of maximal subgroups contained in \cite{atlas} or in \cite{raw}; the lists are complete with the exception of the Monster group, in which case the list of possible maximal subgroups is known.
The upper bound for the \drank{} of each sporadic group is recorded in \cref{spo} (which also lists the groups $\L_n(2)$ for comparison).

If $K$ is a sporadic group of \drank{} smaller than $n$, then $K$ is not the smallest quotient of $\SAut(F_n)$ by \cref{killing Dn} (observing that $K$ is not the smallest quotient of $\SL_n(\Z)$). This observation allows us to rule out all but one sporadic group; the Deucalion is $\Fi_{22}$, and we deal with it by other means.

\begin{table}[h]\centering\begin{tabular}{l || c | r}
			&Bound\\
$K$			&for \drank{}			&Order of $K$\\			
\hline
\hline
$\mathrm{M}_{11}$		&$3$			&7920\\
$\L_4(2)$		&			&21060\\
\hline
$\mathrm{M}_{12}$		&$3$			&95040\\
$\mathrm{J}_1$		&$4$			&175560\\
$\mathrm{M}_{22}$		&$3$			&443520\\
$\mathrm{J}_2$		&$4$			&604800\\
$\L_5(2)$		&			&9999360\\
\hline
$\mathrm{M}_{23}$		&$3$			&10200960\\
$\mathrm{HS}$		&$4$			&44352000\\
$\mathrm{J}_3$		&$4$			&50232960\\
$\mathrm{M}_{24}$		&$3$			&244823040\\
$\mathrm{M^cL}$		&$4$			&898128000\\
$\mathrm{He}$		&$4$			&4030387200\\
$\L_6(2)$		&			&20158709760\\
\hline
$\mathrm{Ru}$		&$5$			&145926144000\\
$\mathrm{Suz}$		&$5$			&448345497600\\
$\mathrm{O'N}$		&$4$			&460815505920\\
$\mathrm{Co}_3$		&$5$			&495766656000\\
$\mathrm{Co}_2$		&$6$			&42305421312000\\
$\Fi_{22}$		&$7$			&64561751654400\\
$\L_7(2)$		&			&163849992929280\\
\hline
$\mathrm{HN}$		&$6$			&273030912000000\\
$\mathrm{Ly}$			&$5$			&51765179004000000\\
$\mathrm{Th}$			&$6$			&90745943887872000\\
$\Fi_{23}$		&$7$			&4089470473293004800\\
$\mathrm{Co}_1$		&$6$			&4157776806543360000\\
$\L_8(2)$		&			&5348063769211699200\\
\hline
$\mathrm{J}_4$		&$7$			&86775571046077562880\\
$\L_9(2)$		&			&699612310033197642547200\\
\hline
$\Fi_{24}'$		&$9$			&1255205709190661721292800\\
$\L_{10}(2)$	&			&366440137299948128422802227200\\
\hline
$\mathrm{B}$			&$10$		&4154781481226426191177580544000000\\
$\L_{11}(2)$	&			&768105432118265670534631586896281600\\
$\L_{12}(2)$	&			&6441762292785762141878919881400879415296000\\
$\L_{13}(2)$	&			&216123289355092695876117433338079655078664339456000\\
\hline
$\mathrm{M}$			&$12$ &808017424794512875886459904961710757005754368000000000\\
\hline
\end{tabular} \caption{Upper bounds for the $D'$-ranks of the sporadic groups.}
\label{spo}
\end{table}

\begin{lem}
\label{fischer}
Every homomorphism $\phi \colon \SAut(F_7) \to \Fi_{22}$ is trivial.
\end{lem}
\begin{proof}
 In the ATLAS~\cite{atlas} we see that there is a single conjugacy class of elements of order 5 in $\Fi_{22}$ denoted $5A$; moreover, the centraliser of an element $x \in 5A$ is of cardinality 600. We also see that $\Fi_{22}$ contains a copy of $S_{10}$, and we may without loss of generality assume that $x$ is a 5-cycle in $S_{10}$. But then the centraliser of $x$ inside $S_{10}$ is $5 \times S_5$, which is already of order 600, and thus coincides with the centraliser of $x$ in $\Fi_{22}$.

Let $\xi$  be the element of order $5$ given by \cref{centralisers}; since its normal closure is $\SAut(F_7)$, its image in $\Fi_{22}$ is not trivial.
Looking at the centraliser of $\xi$, we obtain a homomorphism
\[ \psi \colon \SAut(F_3) \to 5 \times S_5 \]
Since $\SAut(F_3)$ is perfect (\cref{perfect}), the image of $\psi$ must lie within $A_5$. \cref{A_5 not quotient of F_3} tells us that then $\psi$ is trivial. But $\psi$ is a restriction of $\phi$, and so $\phi$ trivialises a transvection, say $\rho_{67}$, and thus $\phi$ is trivial.
\end{proof}

\begin{prop}
\label{main: spo}
Let $n \geqslant 3$ and $K$ be a sporadic simple group. Then $K$ is not the smallest finite quotient of $\SAut(F_n)$.
\end{prop}
\begin{proof}
Let $K$ be a sporadic group, and suppose that it is a smallest finite quotient of $\SAut(F_n)$. We must have
\[
 \vert K \vert \leqslant \vert \L_n(2) \vert
\]
In fact, the inequality is strict, since for each $n$ the group $\L_n(2)$ is not isomorphic to a sporadic group (this is visible in \cref{spo}).
Thus, by \cref{killing Dn}, we see that the epimorphism $\phi \colon \SAut(F_n) \to K$ has to be injective on $D_n'$.
Inspection of \cref{spo} shows that this is only possible for $K=\Fi_{22}$, in which case $n\geqslant 7$. But this is ruled out by \cref{fischer}.
\end{proof}

\section{Algebraic groups and groups of Lie type}

In this section we review the necessary information about algebraic groups over fields of positive characteristic, and the (closely related) finite groups of Lie type.

\subsection{Algebraic groups}
 We begin by discussing connected algebraic groups. Following \cite{Gorensteinetal1998}, we will denote such groups by $\overline K$.
 We review here only the facts that will be useful to us, focusing on simple, semi-simple, and reductive algebraic groups.

Let $r$ be a prime, and let $\Fbar$ be an algebraically closed field of characteristic $r$.
The simple algebraic groups  over $\Fbar$  are classified by the Dynkin diagrams $\A_n$ (for each $n$), $\B_n$ (for $n \geqslant 3$), $\typeC_n$ (for $n\geqslant 2$), $\D_n$ (for $n \geqslant 4$), $\E_n$ (for $n \in \{6,7,8\}$), $\typeF_4$, and $\G_2$. The index of the diagram is defined to be the \emph{rank} of the associated group.

To each Dynkin diagram we associate a finite number of simple algebraic groups; two such groups associated to the same diagram are called \emph{versions}; they become isomorphic upon dividing them by their respective finite centres. Two versions are particularly important: the \emph{universal} (or \emph{simply-connected}) one, which maps onto every other version with a finite central kernel, and the \emph{adjoint} version, which is a quotient of every other version with a finite central kernel.

Every semi-simple algebraic group over $\Fbar$ is a central product of finitely many simple algebraic groups over $\Fbar$. The rank of such a group is defined to be the sum of the ranks of the simple factors, and  is well-defined.

Every reductive algebraic group is a product of an abelian group and a semi-simple group. Its rank is defined to be the rank of the semi-simple factor, and again it is well-defined.

Given an algebraic group $\overline K$, a maximal with respect to inclusion closed connected solvable subgroup of $\overline K$ will be referred to as a \emph{Borel subgroup}. Borel subgroups always exist, and they are all conjugate; hence one can talk about \emph{the} Borel subgroup (up to conjugation). When $\overline K$ is reductive, any closed subgroup thereof containing a Borel subgroup is called \emph{parabolic}.
Let us now state the main tool in our approach towards algebraic groups and groups of Lie type.

\begin{thm}[Borel--Tits~{\cite[Theorem 3.1.1(a)]{Gorensteinetal1998}}]
\label{borel--tits alg}
 Let $\overline K$ be a reductive algebraic group over an algebraically closed field, let $\overline X$ be a
 closed unipotent subgroup, and let $\overline N$ denote the normaliser of $\overline X$ in $\overline K$.
Then there exists a parabolic subgroup $\overline P \leqslant \overline K$ such that $\overline X$ lies in the unipotent radical of $\overline P$, and $\overline N \leqslant \overline P$.
\end{thm}

We will not discuss the other various terms appearing above beyond what is strictly necessary. For our purpose we only need to observe the following.

\begin{rmk}
\label{reductive rmks}
 \begin{enumerate}
  \item The unipotent radical of a reductive group is trivial.
  \item If $\overline K$ is defined in characteristic $r$, then every finite $r$-group in $\overline K$ is a closed unipotent subgroup.
 \end{enumerate}
\end{rmk}

\begin{thm}[Levi decomposition~{\cite[Theorem 1.13.2, Proposition 1.13.3]{Gorensteinetal1998}}]
\label{levi factors alg}
 Let $\overline P$ be a proper parabolic subgroup in a reductive algebraic group $\overline K$.
 \begin{enumerate}
  \item Let $\overline U$ denote the unipotent radical of $\overline P$ (note that $\overline U$ is nilpotent). There exists a subgroup $\overline L \leqslant \overline P$, such that $\overline P = \overline U \rtimes \overline L$.
  \item The subgroup $\overline L$ (the \emph{Levi factor}) is a reductive algebraic group of rank smaller than $\overline K$.
 \end{enumerate}
 \end{thm}

\subsection{Finite groups of Lie type}

Let $r$ be a prime, and $q$ a power thereof.

Any finite group of Lie type $K$ is obtained as a fixed point set of a Steinberg endomorphism of a connected simple algebraic group $\overline K$ defined over an algebraically closed field of characteristic $r$. Such groups have a \emph{type}, which is related to the Dynkin diagram of $\overline K$, and an associated \emph{twisted rank}. As stated in \cref{sec: intro}, the finite groups of Lie type fall into two families: the types $\A_n, \Atwo_n, \B_n,\typeC_n, \D_n$ and $\Dtwo_n$ are called \emph{classical}, and the types $\Btwo_2$, $\Dthree_4$, $\E_6$, $\Etwo_6$, $\E_7$, $\E_8$, $\typeF_4$, $\typeFtwo_4$, $\G_2$ and
$\Gtwo_2$ are called \emph{exceptional}.

The types of the classical groups and their twisted ranks are listed in \cref{dynkinclassical}. Note that $\lceil . \rceil$ denotes the ceiling function. In the case of the exceptional groups, for the groups of types
$\G_2$, $\typeF_4$, $\E_6$, $\E_7$ and $\E_8$ the twisted rank is equal to the rank. Groups of type $\Btwo_2$ or $\Gtwo_2$ have twisted rank $1$, those of type $\Dthree_4$ or $\typeFtwo_4$ have twisted rank $2$ and groups of type $\Etwo_6$ have twisted rank $4$. Groups of types $\Btwo_2$ and $\typeFtwo_4$ are defined only over fields of order $2^{2m+1}$ while groups of type $\Gtwo_2$ are defined only over fields of order $3^{2m+1}$. All groups of all other types are defined in all characteristics.

As was the case with algebraic groups, each type corresponds to a finite number of finite groups (the \emph{versions}), and two such are related by  factoring out the centre as before. The smallest version (in cardinality, say) is called \emph{adjoint} as before; the adjoint version is a simple group with the following exceptions \cite[Chapter 3.5]{atlas}
\begin{alignat*}{4}
\A_1(2) &\cong S_3, & \A_1(3) &\cong A_4, & \typeC_2(2) &\cong S_6, &
\Atwo_2(2) &\cong 3^2 \rtimes Q_8, \\
\G_2(2) &\cong \Atwo_2(3) \rtimes 2,\phantom{xx} & \Btwo_2(2) &\cong 5 \rtimes 4,\phantom{xx} &
\Gtwo_2(3) &\cong \A_1(8) \rtimes 3,\phantom{xx} & \typeFtwo_4(2)
\end{alignat*}
where $Q_8$ denotes the quaternion group of order $8$, and the index $2$ derived subgroup of $\typeFtwo_4(2)$ is simple, known as the Tits group. For the purpose of this paper, we treat the Tits group as  an exceptional finite group of Lie type.

For reference, we also recall the following additional exceptional isomorphisms
\[\A_1(4) \cong \A_1(5) \cong A_5, \; \; \A_1(9) \cong A_6, \; \; \A_1(7) \cong \A_2(2), \; \; \A_3(2) \cong A_8, \; \; \Atwo_3(2) \cong \typeC_2(3)\]

In addition, $\B_n(2^e) \cong \typeC_n(2^e)$ for all $n \geqslant 3$ and $e \geqslant 1$.
We will sometimes abuse the notation, and denote the adjoint version of some type by the type itself.

The adjoint version of a classical group over $q$ comes with a natural projective module over an algebraically closed field in characteristic $r$; the dimensions of these modules are taken from \cite[Table 5.4.C]{KleidmanLiebeck1990} and listed in \cref{dynkinclassical}. Note that these projective modules are irreducible.

\begin{table}
\centering\begin{tabular}{r | c | c | c | c}
			&				&Twisted			&			&Classical\\
Type 		& Conditions 		&rank			& Dimension 	&isomorphism \\ \hline \hline
$\A_n(q)$		&  $n\geqslant 1$ 	& $n$ 			& $n+1$ 		&$\L_{n+1}(q)$\\
$\Atwo_n(q)$	&  $n\geqslant 2$ 	& $\lceil \frac n 2 \rceil$ & $n+1$	&U$_{n+1}(q)$\\
$\B_n(q)$		& $n\geqslant 2$ 	& $n$ 			& $2n+1$ 		&O$_{2n+1}(q)$\\
$\typeC_n(q)$ 	& $n\geqslant 3$ 	& $n$ 			& $2n$ 		&S$_{2n}(q)$\\
$\D_n(q)$ 	& $n\geqslant 4$ 	& $n$ 			& $2n$		&O$_{2n}^+(q)$\\
$\Dtwo_n(q)$	& $n\geqslant 4$ 	& $n-1$ 			& $2n$		&O$_{2n}^-(q)$\\ \hline
\end{tabular}
\caption{The classical groups of Lie type}
\label{dynkinclassical}
\end{table}

A \emph{parabolic} subgroup of $K$ is any subgroup containing the Borel subgroup of $K$, which is obtained by taking an $\alpha$-invariant Borel subgroup in $\overline K$ (where $\alpha$ denotes a Steinberg endomorphism), and intersecting it with $K$. Note that such a Borel subgroup always exists -- in fact, its intersection with $K$ is equal to the normaliser of some Sylow $r$-subgroup of $K$.

\begin{thm}[Borel--Tits~{\cite[Theorem 3.1.3(a)]{Gorensteinetal1998}}]
\label{borel--tits}
 Let $K$ be a finite group of Lie type in characteristic $r$, and let $R$ be a non-trivial $r$-subgroup of $K$. Then there exists a proper parabolic subgroup $P \leqslant K$ such that $R$ lies in the normal $r$-core of $P$, and $N_K(R) \leqslant P$.
\end{thm}

Note that the normal $r$-core of $K$ is trivial.

\begin{thm}[Levi decomposition~{\cite[Theorem 2.6.5(e,f,g), Proposition 2.6.2(a,b)]{Gorensteinetal1998}}]
\label{levi factors}
 Let $P$ be a proper parabolic in a finite group $K$ of Lie type in characteristic $r$.
 \begin{enumerate}
  \item Let $U$ denote the normal $r$-core of $P$ (note that $U$ is nilpotent).  There exists a subgroup $L \leqslant P$, such that $L\cap U = \{1\}$ and $L U = P$.
  \item The subgroup $L$ (the \emph{Levi factor}) contains a normal subgroup $M$ such that $L/M$ is abelian of order coprime to $r$.
  \item The subgroup $M$ is isomorphic to a central  product of finite groups of Lie type (the \emph{simple factors of $L$}) in characteristic $r$ such that the sum of the twisted ranks of these groups is lower than the twisted rank of $K$.
  \item When $K$ is of classical type other than $\Dtwo$ or $\B$, then each simple factor of $L$ is either of the same type as $K$, or of type $\A$. For type $\Dtwo$ we also get factors of type $\Atwo_3$; for type $\B$ we also get factors of type $\typeC_2$.
  \item When $K$ is of classical type other than $\Atwo$ or $\Dtwo$, then the simple factors of $L$ are defined over the same field. The groups $\Atwo(q)$ admit simple factors of $L$ of type $\A(q^2)$, and the groups $\Dtwo(q)$ admit a simple factor of $L$ of type $\A_1(q^2)$.
 \end{enumerate}
 \end{thm}

\section{Groups of Lie type in characteristic $2$}

Because of the special role the involutions $\epsilon_1, \dots, \epsilon_n$ play in the structure of $\Aut(F_n)$, groups of Lie type in characteristic $2$ require a different approach than groups in odd characteristic. The strategy is to look at the centraliser of $\epsilon_n$ in $\Aut(F_n)$, note that it contains $\Aut(F_{n-1})$, and then use the Borel-Tits theorem (\cref{borel--tits}) for its image. The same strategy works for reductive algebraic groups in characteristic $2$.

\subsection{Uniqueness of the map \texorpdfstring{$\SAut(F_n) \to \L_n(2)$}{SAut(F_n) to L_n(2)}}
Before we proceed to the main part of this section, we will investigate maps $\SAut(F_n) \to \L_n(2)$, with the aim of showing that there is essentially only one such non-trivial map.

\begin{lem}[{\cite[Proposition 5.3.7]{KleidmanLiebeck1990}}]
\label{kleidmanliebeck}
Let $A_n$ be the alternating group of degree $n$ where $3 \leqslant n \leqslant 8$. The degree $R_p(A_n)$ of the smallest nontrivial irreducible projective representation of $A_n$ over a field of characteristic $p$ is as given in \cref{smallaltrep}. If $n \geqslant 9$, then the degree of the smallest nontrivial projective representation of $A_n$ is at least $n-2$. \end{lem}

\begin{table}[h]\centering\begin{tabular}{c | c c c c c c}
$n$             &$R_2(A_n)$     &$R_3(A_n)$     &$R_5(A_n)$     &$R_7(A_n)$\\
\hline
5               &2                              &2                              &2                              &2\\
6               &3                              &2                              &3                              &3\\
7               &4                              &4                              &3                              &4\\
8               &4                              &7                              &7                              &7\\ \hline
\end{tabular} \caption{\label{smallaltrep}}
\end{table}

\begin{rmk}
\label{wagner rmk}
Moreover, the result of Wagner \cite{Wagner1976} tells us the following: Assume that the characteristic is $2$ and that $n\geqslant 9$. Then the smallest non-trivial $A_n$-module appears in dimension $n-1$ when $n$ is odd and $n-2$ when $n$ is even, and is unique. This module appears as an irreducible module in our group $D'_n = 2^{n-1} \rtimes A_n$, where for $n$ even we take a quotient by $\langle \delta \rangle$.
\end{rmk}

\begin{lem}
\label{uniqueness lemma}
Let $n\geqslant 4$.
Let $\phi \colon D_n' \to \L_m(2) = \GL(V)$ be a homomorphism.
\begin{enumerate}
 \item When $m<n-1$ then $\phi(2^{n-1})$ is trivial.
 \item Suppose that $m=n-1$ and $\phi(2^{n-1})$ is non-trivial. Then $n$ is even, $\phi(\delta) = 1$ and we can choose a basis of $V$ in such a way that either for $i<n-1$ the element $\phi(\epsilon_i \epsilon_{i+1})$ is given by the elementary matrix $E_{1i}$, that is the matrix equal to the identity except at the position $(1i)$, or each element $\phi(\epsilon_i \epsilon_{i+1})$ is given by the elementary matrix $E_{i1}$.
 \item When $m=n$ and we additionally assume that $\phi$ is injective and that when $n=8$ the representation $\phi\vert_{A_8}$ is the $8$-dimensional permutation representation, then we can choose a basis of $V$ in such a way that either for each $i<n-2$ the element  $\phi(\epsilon_i \epsilon_{i+1})$ is given by the elementary matrix $E_{1i}$, or each element $\phi(\epsilon_i \epsilon_{i+1})$ is given by the elementary matrix $E_{i1}$.
\end{enumerate}
\end{lem}
\begin{proof}
Fix $n$, and proceed by induction on $m$. Clearly $m>1$.

Consider the subgroup $V \rtimes \phi(2^{n-1}) < V \rtimes \GL(V)$. It is a $2$-group, hence it is nilpotent, and therefore it has a non-trivial centre $Z$. Since $\phi(2^{n-1})$ acts faithfully, we have $Z \leqslant V$ as a subgroup, and hence also as a $2$-vector subspace.
Clearly,
 \[
  Z = \{ v \in V \mid \phi(\xi)(v) = v \textrm{ for all } \xi \in 2^{n-1} \}
 \]
and therefore $Z$ is preserved setwise by $\phi(D'_n)$, as $2^{n-1}$ is a normal subgroup of $D'_n$.

Suppose that $\dim Z \leqslant \dim V/Z = m- \dim Z$.
If $n\geqslant 5$, this implies that $Z$ is a trivial $A_n$-module, as the smallest non-trivial module of $A_n$ over the field of two elements is of dimension at least $4$ (see \cite{atlas}).


When $n=4$ we could have $m=n$ and $\dim Z = 2$, in which case $Z$ does not have to be a trivial $A_4$-module. But
in this case we have
\[\GL(Z) \cong \GL(V/Z) \cong \L_2(2) \cong S_3\] and every homomorphism $D'_4 \to S_3$ has $2^3 \rtimes V_4$ in its kernel, where
$V_4$ denotes the Klein four-group. But then
$\phi$ takes $2^3 \rtimes V_4$  to an abelian group of matrices which differ from the identity only in the top-right $2 \times 2$ corner. Thus $\phi(\delta) = \phi([\epsilon_1 \epsilon_2, \sigma_{13} \sigma_{24}]) = 1$, contradicting the injectivity of $\phi$.

We may therefore assume that $Z$ is a trivial $A_n$ module even when $n=4$.

\smallskip
Suppose that $m<n-1$. Then, by the inductive hypothesis, we know that the action of $2^{n-1}$ on $V/Z$ is trivial; it is also trivial on $Z$ by construction.
Hence $\phi(D_n')$ is a subgroup of
\[
 2^{\dim Z (m -\dim Z)} \rtimes \GL(V/Z)
\]
and $\phi$ takes $2^{n-1}$ into the $2^{\dim Z (m -\dim Z)}$ part. But this subgroup cannot contain $2^{n-1}$ as an $A_n$-module, since as a $\GL(V/Z)$-module it is a direct sum of $(m -\dim Z)$-dimensional modules, and $\dim Z \geqslant 1 $. This shows that $\phi$ is not injective on $2^{n-1}$. But the only subgroup of $2^{n-1}$ which can lie in $\ker \phi$ is $\langle \delta \rangle$, and therefore if $\phi(2^{n-1})$ is not trivial, then we need to be able to fit a $(n-2)$-dimensional module into  $2^{\dim Z (m -\dim Z)}$. This is impossible when $m<n-1$, and so (1) follows.

All of the above was conducted under the assumption that $\dim Z \leqslant m - \dim Z$. If this is not true, then we take the transpose inverse of $\phi$; for this representation the inequality is true, and the kernel of this representation coincides with the kernel of $\phi$.

When $m=n-1$ then we use the same proof to show (2) -- it is clear that we can change the basis of $V$ if necessary to have each $\phi(\epsilon_i \epsilon_{i+1})$ as required.

\smallskip
In case (3) we immediately see that $\dim Z=1$. If $2^{n-1}$ does not act trivially on $V/Z$ then we apply (2). Since now $\phi$ is injective, it must take $2^{n-1}$ to the subgroup of $\L_n(2)$ generated by  $E_{ji}$ with $j\in \{1,2\}$ and $i>j$.
Suppose that for some $i$ we have
\[
 \phi(\epsilon_i \epsilon_{i+1}) = E_{12} + E_{2(i+1)} +M
\]
where $M \in \langle \{ E_{1i} \mid i>2 \} \rangle$. Then $\phi(\epsilon_i \epsilon_{i+1})$ is of order $4$, which is impossible. So
\[
 \phi({2^{n-1}}) \leqslant \langle \{ E_{ji} \mid j\in \{1,2\}, i>2 \} \rangle \cong 2^{n-2} \oplus 2^{n-2}
\]
as an $A_n$-module, which contradicts injectivity of $\phi$. Therefore $2^{n-1}$ acts trivially on $V/Z$, and the result follows as before.
\end{proof}

\begin{rmk}
\label{uniqueness lemma rmk}
 In fact, for $n=3$ we can obtain identical conclusions, with the exception that in (3) we may need to postcompose $\phi$ with an outer automorphism of $\L_3(2)$. To see this note that in (1) we have $\L_m(2) = \L_1(2) = \{1\}$; in (2) we have
 $\L_m(2) = \L_2(2) \cong S_3$, and every map from $\D'_3 \cong A_4$ to $S_3$ has $2^2 \cong V_4$ in the kernel. For (3) we see that $\L_m(2) = \L_3(2)$ contains exactly two conjugacy classes of $A_4 \cong D_3'$, and these are related by an outer automorphism of $\L_3(2)$ (see \cite{atlas}). Thus, up to postcomposing $\phi$ with an outer automorphism of $\L_3(2)$, we may assume that $\phi$ maps the involutions $\epsilon_i \epsilon_j$ in the desired manner.
\end{rmk}

We are now ready for our uniqueness result.

\begin{prop}
\label{n=3 char 2}
\label{uniqueness}
 Let $n \geqslant 3$ and $m \leqslant n$ be integers. If
 \[
\phi \colon \SAut(F_n) \to \L_m(2)
 \]
 is a non-trivial homomorphism, then $m=n$ and $\phi$ is equal to the natural map $\SAut(F_n) \to \L_n(2)$  postcomposed with an automorphism of $\L_n(2)$.
\end{prop}
\begin{proof}
Assume that $n\geqslant 3$.
Observe that if $\phi$ is not injective on $D'_n$ then we are done by \cref{killing Dn} -- one has to note that when $\phi(\delta)=1$ then we know that $\phi$ factors through $\SL_n(\Z)$, and so we know (from the Congruence Subgroup Property) that every non-trivial map $\SL_n(\Z) \to \L_n(2)$ factors through the natural such map. We will assume that $\phi$ is injective on $D_n'$.

We apply \cref{uniqueness lemma} (and \cref{uniqueness lemma rmk} when $n=3$); for $n=8$ we consider $\phi(A_9)$ -- by \cref{wagner rmk}, $\phi$ must be the unique $8$-dimensional representation, and so as an $A_8$-module $V$ is the natural permutation representation.
Up to possibly taking the transpose inverse of $\phi$, we see that $m=n$, and
\[
 \phi(\epsilon_i \epsilon_{i+1}) = E_{1i}
\]
Let $Z$ denote the subspace of $V$ generated by the first basis vector; note that $Z$ is precisely the centraliser in $V$ of $\phi(2^{n-1})$ and coincides with the commutator $[V,\phi(\xi)]$ for every $\xi \in 2^{n-1} \smallsetminus \{1\}$. 
Thus, if a matrix in $\L_n(2)$ commutes with any non-zero sum of matrices $E_{1i}$, then it preserves $Z$.

 The group $\SAut(F_n)$ is generated by transvections, and each of them commutes with some $\epsilon_i \epsilon_{j}$ as $n\geqslant 4$, and so $\SAut(F_n)$ preserves $Z$. Thus we have a representation
\[
 \SAut(F_n) \to \GL(V/Z) \cong \L_{n-1}(2)
\]
and such a representation is trivial, or $n$ is even and the representation has $\delta$ in its kernel by \cref{uniqueness lemma}. But then it factors through $\SL_n(\Z)$, and therefore must be trivial, since the smallest quotient of $\SL_n(\Z)$ is $\L_n(2)$.

Therefore $\phi$ takes $\SAut(F_n)$ to $2^{n-1}$, and hence must be trivial.
\end{proof}

\subsection{The case of small $n$}

We now proceed to the main discussion.
We start by looking at small values of $n$. These considerations will form the base of our induction.

\begin{lem}
\label{maps F_3 to L_2}
Let $\Fbar$ be an algebraically closed field of characteristic $2$.
Every homomorphism $\phi \colon \SAut(F_3) \to \L_2(\Fbar) = \PSL_2(\Fbar)$ is trivial.
\end{lem}
\begin{proof}
We start by observing that $\PSL_2(\Fbar) = \SL_2(\Fbar)$, since the only element in $\Fbar$ which squares to $1$ is $1$ itself.

Suppose first that $\epsilon_1 \epsilon_2$ lies in the kernel of $\phi$. Then $\phi$ descends to a map
\[\L_3(2) \to \L_2(\Fbar)\]
by \cref{killing Dn}. Since $\L_3(2)$ is simple, this map is either faithful or trivial. But it cannot be faithful, since the upper triangular matrices in $\L_3(2)$ form a 2-group (the dihedral group of order $8$) which is nilpotent of class 2, whereas every non-trivial 2-subgroup of $\L_2(\Fbar)$ is abelian. (Alternatively, one can use the fact that $\L_3(2)$ has no non-trivial projective representations in dimension $2$ in characteristic $2$, as can be seen from the $2$-modular Brauer table which exists in GAP.)

Hence we may assume that $\phi(\epsilon_1 \epsilon_2) \neq 1$. Consider the 2-subgroup of $\L_2(\Fbar)$ generated by $\phi(\epsilon_i \epsilon_j)$ with $1 \leqslant i,j \leqslant 3$ (it is isomorphic to $2^2$). As before, up to conjugation, this subgroup lies within the unipotent subgroup of upper triangular matrices with ones on the diagonal.
Now a direct computation shows that the matrices which commute with
\[
\left( \begin{array}{cc} 1 & x \\ 0 & 1 \end{array} \right)
\]
with $x \neq 0$ are precisely the matrices of the form
\[
\left( \begin{array}{cc} \ast & \ast\\ 0 & \ast \end{array} \right)
\]
In particular, this implies that if an element of $\L_2(\Fbar)$ commutes with $\phi(\epsilon_1 \epsilon_2)$, then it also commutes with $\phi(\epsilon_2 \epsilon_3)$. This applies to $\phi(\sigma_{12} \epsilon_3)$, and so
\[
\phi(\epsilon_2 \epsilon_3) = \phi(\epsilon_2 \epsilon_3)^{\phi(\sigma_{12} \epsilon_3)}  =\phi(\epsilon_1 \epsilon_3)
\]
and so $\phi(\epsilon_1 \epsilon_2) = 1$, contradicting our assumption.
\end{proof}

\begin{lem}
\label{maps to lie(2) of small rank basic case}
Every homomorphism $\phi$ from $\Aut(F_3)$ to a finite group of Lie type in characteristic $2$ of twisted rank $1$ has abelian image.
\end{lem}
\begin{proof}
The groups we have to consider as targets here are the versions of $\A_1(q)$, $\Atwo_2(q)$, and $\Btwo_2(q)$, where $q$ is a power of $2$ (where the exponent is odd in type $\Btwo_2$). Since $\SAut(F_3)$ is perfect, and we claim that it has to be contained in the kernel of our homomorphism, we need only look at the adjoint versions.

For type $\A_1$ the result follows from \cref{maps F_3 to L_2}, since $\L_2(q) \leqslant \L_2(\Fbar)$ with $\Fbar$ algebraically closed and of characteristic $2$.
The simple group of type $\Btwo_2(q)$ has no elements of order $3$ (this can easily be seen from the order of the group), and so $\SAut(F_3)$ lies in  the kernel of the homomorphism by \cref{killing Dn}.

We are left with the type $\Atwo_2(q)$. In this case we observe that, up to conjugation, there are only two parabolic subgroups of $K = \Atwo_2(q)$, namely $K$ itself and a Borel subgroup $B$.

Suppose first that $\epsilon_3$ has a non-trivial image in $K$.
By \cref{borel--tits}, the image of the centraliser of $\epsilon_3$ in $\Aut(F_n)$ lies in $B$.
The Borel subgroup $B$ is a semi-direct product of the unipotent subgroup by the torus. The torus contains no elements of order $2$. Moreover, the only elements of order $2$ in the unipotent subgroup lie in its centre -- this can be verified by a direct computation with matrices.

The centraliser of $\epsilon_3$ in $\Aut(F_3)$ contains $\Aut(F_2)$, which is generated by involutions $\epsilon_1, \epsilon_2, \rho_{12} \epsilon_2$ and $\rho_{21} \epsilon_1$. Thus the image of $\Aut(F_2)$ lies in the centre of the unipotent subgroup of $B$, which is abelian. Therefore, we have \[\phi(\rho_{12}) = \phi({\rho_{12}}^{\epsilon_1 \epsilon_2 \sigma_{12}}) =  \phi(\lambda_{21})\]
and therefore
\[\phi({\rho_{13}})^{-1} = [\phi({\rho_{12}})^{-1}, \phi({\rho_{23}})^{-1}] = [\phi({\lambda_{21}})^{-1}, \phi({\rho_{23}})^{-1}] =1 \]
This trivialises the subgroup $\SAut(F_3)$ as claimed.

Recall that we have assumed that $\epsilon_3$ is not in the kernel of $\phi$; when it is, then the homomorphism factors through $\L_3(2)$, which is simple and not a subgroup of $\Atwo_2(q)$ whenever $q$ is a power of $2$ -- this can be seen by inspecting the maximal subgroups of $\Atwo_2(q)$ \cite{bhrd}.
\end{proof}

\subsection{The main result}

\begin{thm}
\label{rigidity for Aut(F_n)}
\label{main thm char 2 lie type}
Let $n \geqslant 3$. Let $K$ be a finite
group of Lie type in characteristic $2$ of twisted rank less than $n-1$, and let $\overline K$ be a reductive algebraic group over an algebraically closed field of characteristic $2$ of rank less than $n-1$. Then any homomorphism $\Aut(F_n) \to K$ or $\Aut(F_n) \to \overline K$ has abelian image, and any homomorphism $\SAut(F_{n+1}) \to K$ or $\SAut(F_{n+1}) \to \overline K$ is trivial.
\end{thm}
\begin{proof}
We start by looking at the finite group $K$, and a homomorphism \[\phi \colon \Aut(F_n) \to K\]

Since $\SAut(F_n)$ is perfect and of index $2$ in $\Aut(F_n)$, we may without loss of generality divide $K$ by its centre; we may also assume that $K$ is not solvable.

Our proof is an induction on $n$. The base case ($n=3$) is covered by \cref{maps to lie(2) of small rank basic case}.
In what follows let us assume that $n>3$.

We claim that $\phi(\SAut(F_{n-1}))$ lies in a proper parabolic subgroup $P$ of $K$.
If $\phi(\epsilon_n )$ is central then
$\phi(\epsilon_{n-1} \epsilon_n)$ is trivial, since $\epsilon_n$ and $\epsilon_{n-1}$ are conjugate. Thus $\phi$ factors through
\[\Aut(F_n) \to \L_n(2)\]
by \cref{killing epsilon}. Let $\eta \colon \L_n(2) \to K$ denote the induced homomorphism.

The group $\L_n(2)$ contains $\L_{n-1}(2)$ inside a proper parabolic subgroup which normalises a non-trivial $2$-group $G$.
This $2$-group contains an elementary matrix, and so if $\eta(G)$ is trivial, then so is every elementary matrix in $\L_n(2)$, and therefore $\eta$ is trivial (as $\L_n(2)$ is generated by elementary matrices). This proves the claim.

Now let us assume that $G$ has a non-trivial image in $K$. Thus, by \cref{borel--tits}, the normaliser of $G$ in $\L_n(2)$ is mapped by $\eta$ into a proper parabolic subgroup $P$. Clearly, we may choose $G$ so that it is normalised by the image of $\Aut(F_{n-1})$ in  $\L_n(2)$. This way we have shown that $\phi(\Aut(F_{n-1}))$ lies in $P$.

Now assume that $\phi(\epsilon_n)$ is not central, and so in particular not trivial. We conclude, using \cref{borel--tits}, that $\phi(\Aut(F_{n-1}))$ lies in a parabolic $P$ inside $K$ such that $P\neq K$. Hence we have
\[
\phi(\Aut(F_{n-1})) \leqslant P < K
\]
irrespectively of what happens to $\epsilon_n$, which proves the claim.

Consider the induced map $\psi \colon \Aut(F_{n-1}) \to P/U \cong L$ (using the notation of \cref{levi factors}). Note that in fact the image of $\psi$ lies in $M$, since $L/M$ is abelian and contains no element of order 2. Now $M$ is a central product of finite groups of Lie type in characteristic $2$, where the sum of the twisted ranks is lower than that of $K$. Thus, using the projections, we get maps
from $\Aut(F_{n-1})$ to finite groups of Lie type of twisted rank less than $n-2$. By the inductive assumption all such maps have abelian image, and so the image of $\Aut(F_{n-1})$ in $M$ is abelian. This forces $\phi$ to contain $\SAut(F_{n-1})$ in its kernel, and the result follows, since $U$ is nilpotent and $\SAut(F_{n-1})$ is perfect as $n\geqslant 4$.

\smallskip
Now let us look at a homomorphism $\phi \colon \Aut(F_n) \to \overline K$. We proceed as above; the base case ($n=3$) is covered by \cref{maps F_3 to L_2}.

We claim that, as before, $\phi(\Aut(F_{n-1}))$ is contained in a proper parabolic subgroup $\overline P$ of $\overline K$. This is proved exactly as before using \cref{borel--tits alg}, except that now we use the fact that every finite $2$-group in $\overline K$ is a closed unipotent subgroup. Note that $\overline P$ is a proper subgroup, since $\overline K$ is reductive, and thus its unipotent radical is trivial.

Again as before we look at the induced map $\psi \colon \Aut(F_{n-1}) \to \overline P/ \overline U \cong \overline L$.
By \cref{levi factors alg}, the group $\L$ is reductive of lower rank, and so $\psi$ has abelian image by induction. But then $\phi\vert_{\Aut(F_{n-1})}$ has solvable image, and so $\phi(\SAut(F_{n-1}))= \{1\}$. Therefore
\[
 \phi(\SAut(F_{n}))= \{1\}
\]
as well, and the image of $\phi$ is abelian.

\smallskip
The statements for $\SAut(F_{n+1})$ follow from observing that the natural embedding $\SAut(F_{n}) \into \SAut(F_{n+1})$ extends to an embedding $\Aut(F_{n}) \into \SAut(F_{n+1})$, where we map an element $x \in \Aut(F_n)$ of determinant $-1$ to $x \epsilon_{n+1}$. When this copy of $\Aut(F_n)$ has an abelian image under a homomorphism, then the homomorphism is trivial on $\SAut(F_{n})$, and hence on the whole of $\SAut(F_{n+1})$.
\end{proof}

\begin{thm}
\label{thm char 2}
Let $n \geqslant 8$.
Let $K$ be a finite simple group of Lie type in characteristic $2$ which is a quotient of $\SAut(F_n)$. Then either $\vert K \vert > \vert \L_n(2) \vert$, or $K = \L_n(2)$ and $\phi$ is obtained by postcomposing the natural map $\SAut(F_n) \to \L_n(2)$ by an automorphism of $\L_n(2)$.
\end{thm}
\begin{proof}
By \cref{main thm char 2 lie type}, $K$ is of twisted rank at least $n-2$.

Since $n\geqslant 8$, by \cref{computation orders char 2} we see that
all the finite simple groups of Lie type in characteristic $2$ and twisted rank at least $n-2$ are larger than $\L_n(2)$, with the exception of $\A_{n-2}(2)$ and $\A_{n-1}(2)$. \cref{n=3 char 2} immediately tells us that $K = \L_n(2)$ and $\phi$ is as claimed.
\end{proof}

\section{Classical groups in odd characteristic}
We now turn to the classical groups in odd characteristic. Contrary to the situation in characteristic $2$, our approach here is to develop and make use of the representation theory of $\Aut(F_n)$ in small dimensions over odd characteristics.

In the case that $K$ is defined over the field of three elements however, for technical reasons, it is still advantageous to use the Borel--Tits theorem, and this is where we start.

\subsection{Field of $3$ elements}
\label{sec: char 3}
In this subsection we will show that finite groups of Lie type over the field of $3$ elements are not the smallest quotients of $\SAut(F_n)$. To this end, we will
use Borel--Tits in characteristic 3.
To do this, we need to find suitable elements of order $3$ in $\SAut(F_n)$.

Let $\gamma = \epsilon_{n-1} \epsilon_{n} {\lambda_{(n-1)n}}^{-1} \rho_{n (n-1) }$. A direct computation immediately shows that $\gamma$ is of order $3$. Also, the centraliser of $\gamma$ in $\SAut(F_n)$ contains $\SAut(F_{n-2})$. In fact, $\gamma$ is the element constructed in \cref{centralisers}. We define it here algebraically, since it allows us to easily show the following.
\begin{lem}
 \label{killing t}
 Let $n\geqslant 4$.
 The normal closure of $\gamma$ inside $\SAut(F_n)$ is the whole of $\SAut(F_n)$.
\end{lem}
\begin{proof}
 Let $C$ denote the normal closure. Then
 \begin{align*}
  {\rho_{n 1}}^{-1} C &= [{\rho_{n (n-1)}}^{-1},{\rho_{(n-1) 1}}^{-1}] C \\
  &= [\epsilon_{n-1} \epsilon_{n} {\lambda_{(n-1)n}}^{-1},{\rho_{(n-1) 1}}^{-1}] C \\
  &= {\lambda_{(n-1) 1}}{\rho_{(n-1) 1}} C
 \end{align*}
 where the last equality follows by expanding the commutator. Now
 \begin{align*}
  {\rho_{2 1}}^{-1} C &= [{\rho_{2 n}}^{-1},{\rho_{n 1}}^{-1}] C \\
  &= [{\rho_{2 n}}^{-1},{\lambda_{(n-1) 1}}{\rho_{(n-1) 1}}] C \\
  &=  C
 \end{align*}
 and we are done.
\end{proof}

\begin{table}
\caption{Some conjugacy classes in $\typeC_2(3)$}
\label{cc table}
\begin{center}
\begin{tabular}{c | c c c c c c c c c c c c c c c c c}
Class                                   &$2A$   &$2B$   &$3C$   &$3D$   &$4A$   &$4B$   &$5A$   &$6E$   &$6F$\\ \hline
$\vert x^G \cap C_G(x) \vert$   &13             &22             &6              &12             &8              &4              &4              &2              &2\\
\end{tabular}
\end{center}
\label{2a33}
\end{table}%

Now that we have a suitable element of order $3$ in $\SAut(F_n)$ at our disposal, we will be able to use Borel--Tits theorem. First we need to (as usual) look into the small values of $n$ to be able to start the induction.

\begin{lem}
\label{char 3 even base case}
Every homomorphism $\phi \colon \SAut(F_4) \to K$, where $K$ is a finite group of Lie type of type $\A_2(3)$ or $\typeC_2(3)$, is trivial.
\end{lem}
\begin{proof}
Since $\SAut(F_4)$ is perfect, we may assume that $K$ is simple. If $K$ is of type $\A$, then $K \cong \L_3(3)$
which has no element of order $5$ -- this follows immediately from the order of the group. But then $\phi$ trivialises the five cycle in $A_5$, and so $\phi$ is trivial by \cref{killing Dn}.

Suppose that $K$ is of type $\typeC$. 
We are now going to argue as in  the proof of \cref{A_5 not quotient of F_3}.
Consider the set of transvections
\[
 T = \{ {\rho_{ij}}^{\pm 1}, {\lambda_{ij}}^{\pm 1} \}
\]
Recall that any two elements in $T$ are conjugate in $\SAut(F_4)$ (by \cref{conjugate transvections}).
Let $ C_T(\rho_{12})$ denote the set of elements in $T$ which commute with $\rho_{12}$.
There are exactly $24$ elements in  $C_T(\rho_{12})$, namely
\[
 \{{\rho_{12}}^{\pm1}, {\lambda_{12}}^{\pm1}, {\lambda_{13}}^{\pm1}, {\lambda_{14}}^{\pm1},  {\rho_{32}}^{\pm1}, {\lambda_{32}}^{\pm1}, {\rho_{42}}^{\pm1}, {\lambda_{42}}^{\pm1}, {\rho_{34}}^{\pm1}, {\lambda_{34}}^{\pm1}, {\rho_{43}}^{\pm1}, {\lambda_{43}}^{\pm1}\}
\]

\cref{cc table} lists every conjugacy class in $K$ of elements conjugate to their inverse, as can be computed in GAP; it also lists the number of elements in the conjugacy class which commute with a fixed representative of the class.

Note that if $\phi(\rho_{12})$ is an involution, then a direct computation with GAP reveals that $\phi$ factors through
\[
 \SAut(F_4) / \langle \! \langle {\rho_{12}}^{2} \rangle \! \rangle \cong 2^4 \rtimes \L_4(2)
\]
(Note that an analogous statement is true for $n=3$, but for large enough $n$ the quotient is infinite, as shown in \cite{BridsonVogtmann2003}.)
The group $K$ is simple and non-abelian, and so $\phi$ factors through $\L_4(2)$. Hence $\phi$ is trivial, as $\L_4(2)$ is simple and not isomorphic to $K$.

We may thus assume that $\phi(\rho_{12})$ is not an involution. Inspecting \cref{cc table} we see that
there are at most $12$ elements in the conjugacy class of $\phi(\rho_{12})$ which commute with $\phi(\rho_{12})$. Thus there exist two elements in $ C_T(\rho_{12})$ which get identified under $\phi$. Without loss of generality we may assume that we have
\[
 \phi(\rho_{12}) = \phi({x_{ij}}^{\pm1})
\]
where $x$ is either $\rho$ or $\lambda$, and ${x_{ij}}^{\pm1} \not\in \{ \rho_{12}, {\rho_{12}}^{-1}\}$.

If $j>2$, take $k$ such that $k \in \{2,3,4\} \s- \{i,j \}$. Now
\[
 \phi({x_{ik}}^{-1}) = \phi([{x_{ij}}^{-1},{x_{jk}}^{-1}]) = [\phi(\rho_{12})^{\mp 1},\phi(x_{jk})^{-1}] = 1
\]
and so $\phi$ is trivial. Let us assume that $j \leqslant 2$.

Similarly, if $i>1$,  take $k\in \{3,4\} \s- \{i\}$. Now
\[
 \phi({x_{kj}}^{-1}) = \phi([{x_{ki}}^{-1},{x_{ij}}^{-1}]) = [\phi(x_{ki})^{-1},\phi(\rho_{12})^{\mp 1}] = 1
\]
and so $\phi$ is trivial.

We are left with the case $(i,j) = (1,2)$ and $x = \lambda$. If ${x_{ij}}^{\pm 1}=\lambda_{12}$ then  $\phi$ factors through $\SL_4(\Z)$, since adding the relation $\rho_{12} {\lambda_{12}}^{-1}$ takes the Gersten's presentation of $\SAut(F_n)$ to the Steinberg's presentation of $\SL_n(\Z)$. But we know all the finite simple quotients of $\SL_4(\Z)$, and $K$ is not one of them. Hence $\phi$ is trivial.

We are left with the case ${x_{ij}}^{\pm 1}={\lambda_{12}}^{-1}$. Gersten's presentation contains the relation
\[
 (\rho_{12} {\rho_{21}}^{-1} \lambda_{12})^4
\]
Using the relation $\rho_{12} {\lambda_{12}}$ gives
\[
 (\rho_{12} {\rho_{21}}^{-1} {\rho_{12}}^{-1})^4
\]
  which is equivalent to ${\rho_{21}}^4$. Thus $\phi(\rho_{21})$, and hence also $\phi(\rho_{12})$, has order $4$. Inspecting \cref{cc table} again we see that in fact we have at least three elements in  $C_T(\rho_{12})$ which coincide under $\phi$, and so, without loss of generality, there exists ${x_{ij}}^{\pm1} \not\in \{ \rho_{12}, {\rho_{12}}^{-1}, {\lambda_{12}}^{-1} \}$ such that
  \[
\phi(\rho_{12}) = \phi({x_{ij}}^{\pm1})
  \]
Thus we are in one of the cases already considered.
\end{proof}

\begin{lem}
\label{char 3 odd base case}
Every homomorphism $\phi \colon \SAut(F_5) \to K$ where $K$ is a finite group of Lie type of type $\A_3(3), \Atwo_3(3)$ or $\typeC_3(3)$ is trivial.
\end{lem}
\begin{proof}
As always, we assume that $K$ is simple. The simple group $\A_3(3) \cong \L_4(3)$ contains two conjugacy classes of involutions \cite{atlas} where they are denoted $2A$ and $2B$. The ATLAS also gives the order of their centralisers. The centraliser of an involution in class $2B$ has order $1152=2^73^2$ and hence is solvable by Burnside's $p^aq^b$-Theorem. The structure of the centraliser of an involution in class $2A$ is given in \cite{atlas} and is isomorphic to
\[
 (4 \times A_6) \rtimes 2
\]
Consider $\epsilon_4 \epsilon_5 \in \SAut(F_5)$. If $\phi(\epsilon_4 \epsilon_5) = 1$ then $\phi$ factors through $\L_5(2)$ (by \cref{killing Dn}), which is simple and non-isomorphic to $\A_2(3)$. This trivialises $\phi$.

If $\phi(\epsilon_4 \epsilon_5) \neq 1$ then $\phi$ maps $\SAut(F_3)$ (which centralises $\epsilon_4 \epsilon_5$) to either a solvable group, or to $(4 \times A_6) \rtimes 2$. In both cases we have $\SAut(F_3) \leqslant \ker \phi$, as $\SAut(F_3)$ is perfect and has no non-trivial homomorphisms to $A_6$ by \cref{A_5 not quotient of F_3}.
This trivialises $\phi$.

The simple group $\Atwo_3(3)$ has a single conjugacy class of involutions, denoted $2A$ in \cite{atlas}, and the centraliser of an involution in this class again has order 1152, hence it is solvable, and so we argue as before.

The conjugacy classes of maximal subgroups of the simple group $\typeC_3(3)$ are known \cite[pg.113]{atlas}. By inspection we see that it does not contain $D_5'$ and so $\phi$ factors through $\L_5(2)$ by \cref{killing Dn}.
But $\L_5(2)$ contains an element of order $31$, whereas $\typeC_3(3)$ does not. Thus $\phi$ is trivial.
\end{proof}

With the base cases covered, we are ready for the main result in this section.

\begin{lem}
\label{char 3 small}
Let $n\geqslant 4$.
Every homomorphism $\phi \colon \SAut(F_n) \to K$ is trivial, where
\begin{enumerate}
 \item $n$ is even, and $K$ is the of type $\A_k(3)$ or $\B_k(3)$  or $\typeC_2(3)$ with $k \leqslant \frac n 2$; or
 \item $n$ is odd, and $K$ is of type $\A_k(3),  \Atwo_3(3), \typeC_k(3), \D_k(3)$ or $\Dtwo_k(3)$ with $k \leqslant \frac {n+1} 2$.
\end{enumerate}
\end{lem}
\begin{proof}
As usual, since $\SAut(F_n)$ is perfect, we may assume that we are dealing with adjoint versions; therefore we will use type to denote its adjoint version.

The proof is an induction; the base case when $n$ is even is covered by \cref{char 3 even base case}, upon noting that for $k=1$ we only have to consider $\A_1(3)$, which is solvable.

When $n$ is odd, the base case consists of the groups $\A_1(3)$, $\A_2(3)$, $\A_3(3)$, $\typeC_2(3)$, $\typeC_3(3)$, 
and $\Atwo_3(3)$.
The first two are subgroups of the third, which is covered by \cref{char 3 odd base case}, and so is $\typeC_3(3)$. The group $\typeC_2(3)$ is covered by \cref{char 3 even base case}.
The remaining group $\Atwo_3(3)$ is again covered by \cref{char 3 odd base case}.

\smallskip
Now suppose that $n>4$. Consider $\phi(\gamma)$. If this is trivial, then we are done by \cref{killing t}. Otherwise, \cref{borel--tits} tells us that $\phi$ maps $\SAut(F_{n-2})$ to a parabolic subgroup $P$ of $K$.
We will now use the notation of \cref{levi factors}.

Let
\[\psi \colon \SAut(F_{n-2}) \to L\]
 be the map induced by taking the quotient $P \to P/U \cong L$. Since $L/M$ is abelian, and $\SAut(F_{n-2})$ is perfect, we immediately see that $\im\psi \leqslant M$.

Suppose that $n$ is even. Then $M$ admits projections onto groups of type $\A_l(3)$ or $\B_l(3)$ with $l<k$ or $\typeC_2(3)$. The inductive hypothesis shows that $\psi(\SAut(F_{n-2}))$ lies in the intersection of the kernels of such projections. But $M$ is a central product of the images of these projections, and so
$\psi$ is trivial. But then $\phi$ trivialises $\SAut(F_{n-2})$, and the result follows.

When $n$ is odd the situation is similar: the group $M$ admits projections to groups of type $\A_l(3)$, $\Atwo_3(3)$, $\typeC_l(3), \D_l(3), \Dtwo_l(3)$ or $\A_1(9)$. The last group is isomorphic to $A_6$, and every homomorphism from $\SAut(F_3)$ to $A_6$ is trivial by \cref{A_5 not quotient of F_3}. The other groups are covered by the inductive hypothesis, and we conclude as before.
\end{proof}

\begin{rmk}
In fact the groups of type $\A_k(3)$ are not quotients of $\SAut(F_n)$ when $k \leqslant n-2$ which will become clear in the following section.
\end{rmk}

\subsection{Representations of $D'_n$}

Having dealt with the field of $3$ elements, we move on to study projective representations of $\SAut(F_n)$ in odd characteristics.
To do this we will first develop some representation theory of the subgroup $D'_n$.

\begin{dfn}
 The action of $D'_n$ on $\Z^n$ obtained by abelianising $F_n$ is the \emph{standard action}. Tensoring $\Z^n$ with a field $\F$ gives us the \emph{standard $D'_n$-module} $\F^n$, and the image of the generators $a_1, \dots, a_n$ in $\F^n$ is the \emph{standard basis}.
\end{dfn}

\begin{dfn}
 Let $\pi$ be a representation of $2^{n-1}$. We set
\[
 E_I = \{v \in V \mid \pi(\epsilon_i \epsilon_j)(v) = (-1)^{\chi_I(i) + \chi_I(j)} v \textrm{ for all } i,j \in N\}
\]
 with $\chi_I$ standing for the characteristic function of $I \subseteq N$.
\end{dfn}

Note that $E_I = E_{N \s- I}$, but otherwise these subspaces intersect trivially.

\begin{lem}
\label{reps of D'n}
Let $n \geqslant 7$.
Let $\pi\colon D'_n \to \GL(V)$ be a linear representation of $D'_n$ over a field of characteristic other than $2$ in dimension $k<2n$, such that there is no vector fixed by all elements $\pi(\epsilon_i \epsilon_j)$. Then $k=n$ and $\pi$ is the standard representation.
\end{lem}
\begin{proof}
The elements $\pi(\epsilon_i \epsilon_j) \in \GL(V)$ are all commuting involutions, and so we can simultaneously diagonalise them (since the characteristic of the ground field is not $2$).
This implies that
\[
 V = \bigoplus_{\vert I \vert \leqslant \frac n 2 } E_I
\]

Note that for each $m\leqslant \frac n 2$, the subgroup $A_n$ acts on $\bigoplus_{\vert I \vert = m} E_I$; such a subspace is also preserved by the subgroup $2^{n-1}$, and so by the whole of $D'_n$. Since there are no vectors fixed by each $\pi(\epsilon_i \epsilon_j)$, we have
\[
 V = \bigoplus_{\vert I \vert > 0 } E_I
\]
The action of $A_n$ permutes the subspaces $E_I$ according to the natural action of $A_n$ on the subsets of $N$. Hence for any $k < \frac n 2$ we have
\[
 \dim \bigoplus_{\vert I \vert = k} E_I = \binom n k \dim E_I
\]
for any $I \subseteq N$ with $\vert I \vert = k$, and for $k = \frac n 2$ (assuming that $n$ is even) we have
\[
 \dim \bigoplus_{\vert I \vert = k} E_I = \frac 1 2 \binom n k \dim E_I
\]

Since $\dim V < 2n$, we conclude that
\[
 V = \bigoplus_{i \in N} E_{\{i\}}
\]
and each $E_{\{i\}}$ is $1$-dimensional. Let us pick a non-zero vector in $E_{\{i\}}$ for each $i$; these vectors form a basis of $V$.

It is immediate that with respect to this basis, the action of $2^{n-1}$ agrees with that of the restriction of the standard representation of $D_n'$ to $2^{n-1}$; moreover, it also shows that for each $\tau \in A_n$ the matrix $\pi(\tau)$ is a monomial matrix obtainable from the matrix given by the standard representation of $D_n'$ by multiplication by a diagonal matrix.

Since $n\geqslant 6$, the setwise stabiliser in $A_n$ of any $E_{\{i\}}$ is simple (as it is isomorphic to $A_{n-1}$), and so we can rescale each vector in our basis so that $\pi(\tau)$ becomes a permutation matrix for each $\tau \in A_n$, and this concludes our proof.
\end{proof}

Recall that $R_p(A_n)$ (occurring in the statement of the following result) denotes the minimal dimension of a faithful projective representation of $A_n$ as in \cite{KleidmanLiebeck1990} (the values of $R_p(A_n)$ are stated in \cref{kleidmanliebeck}).

\begin{prop}
\label{Barbara even}
Let $n\geqslant 8$ be even.
 Let $\pi \colon D'_n \to \PGL(V)$ be a faithful projective representation of dimension less than $ n +R_p(A_n)$ over an algebraically closed field $\Fbar$ of characteristic $p > 2$. Then the projective representation lifts to a representation $\overline \pi \colon D'_n \to \GL(V)$ so that the module $V$ splits as $W \oplus U$ where $W$ is a sum of trivial modules, and $U$ is the standard module of  $D'_n$.
\end{prop}
\begin{proof}
Let $d \in \GL(V)$ be a lift of $\pi(\delta)$. Since $\delta$ is an involution, $d^2$ is central, and so the minimal polynomial of $d$ is
$x^2 - \lambda$ for some $\lambda \in \F^\times$. Since the field $\F$ is algebraically closed and not of characteristic 2, this minimal polynomial has two distinct roots, and so $d$ is diagonalisable. Upon multiplying $d$ by a central matrix we may assume that at least one of the entries in the diagonal matrix of $d$ is 1. Thus all the entries are $\pm 1$, and in particular $d$ is also an involution.

For any $\xi \in D'_n$, let $\overline \xi \in \GL(V)$ denote a lift of $\pi(\xi)$. Since $\delta$ is central in $D'_n$, every $\overline \xi$ either preserves the eigenspaces of $d$, or permutes them. This way we obtain a homomorphism $D'_n \to \Z/2\Z$, which has to be trivial by \cref{closure of An in Dn}. Thus every $\overline \xi$  preserves the eigenspaces of $d$.

Since $\pi(\delta)$ is not trivial (as $\pi$ is faithful), the involution $d$ has a non-trivial eigenspace for each eigenvalue, and the same is true for any other involution lifting $\pi(\delta)$.

Take an eigenspace  of $d$ of dimension less than $n$.
By \cite[Corollary 5.5.4]{KleidmanLiebeck1990}, the projective module obtained by restricting to this eigenspace is not faithful -- in fact, the action of $D_n'$ on $W$ has the whole of $2^{n-1}$ in its kernel.
Therefore, if both eigenspaces of $d$ are of dimension less than $n$, then the kernel of $\pi$ contains an index two subgroup of $2^{n-1}$, and therefore is not trivial. This contradicts the assumption on faithfulness of $\pi$.

We conclude that one of the eigenspaces of $d$, say $U$, has dimension at least $n$. But then the other eigenspace $W$, has dimension less than $R_p(A_n)$, and so the restricted projective $A_n$-module $W$ is trivial.
Hence it is also a trivial projective $D'_n$-module.
The abelianisation of $D'_n$ is trivial, and so for each $\xi \in D_n'$ we may choose $\overline \xi$ so that its restriction to $W$ is the identity matrix.
In this way we obtain a homomorphism $\overline \pi \colon D'_n \to \GL(V)$ by declaring $\overline{\pi}(\xi) = \overline \xi$.
Note that $W$ is a sum of trivial submodules of this representation, and so in particular it is the $(+1)$-eigenspace of $\overline \delta$.

It is easy to see that in fact $V = U \oplus W$ as a $D_n'$-module, since $\overline \xi$  preserves the eigenspaces of $\overline \delta$ for every $\xi \in D_n'$, as remarked above.

Suppose that there is a non-zero vector in $U$ fixed by each $\overline {\epsilon_i \epsilon_j}$. Then it is also fixed by $\overline \delta$, as $\delta = \epsilon_1 \cdots \epsilon_n$ and $n$ is even. But $\overline \delta$ acts as minus the identity on $U$, which is a contradiction. Hence we may apply \cref{reps of D'n} to $U$ and finish the proof.
\end{proof}

\begin{cor}
\label{Barbara}
Let $n\geqslant 8$.
 Let $\pi \colon D'_n \to \PGL(V)$ be a faithful projective representation over an algebraically closed field $\Fbar$ of characteristic $p > 2$ of dimension less than $2\cdot R_p(A_n)$ when $n$ is even or less than $n +  R_p(A_{n-1})-1$ when $n$ is odd. Then the representation lifts to a representation $\overline \pi \colon D'_n \to \GL(V)$, and the module $V$ splits as $W \oplus U$ where $W$ is a sum of trivial modules and $U$ is the standard module of  $D'_n$.
\end{cor}
\begin{proof}
When $n$ is even the result is covered by \cref{Barbara even}; let us assume that $n$ is odd.

 We apply \cref{Barbara even} to two subgroups $P_1$ and $P_2$ of $D'_n$ isomorphic to $D'_{n-1}$, where $P_i$ is the stabiliser of $a_i$ in $D'_n$.

 If $\dim V < n-1$ then we immediately learn that $V$ is the sum of trivial $P_1$ modules, and thus it is also a sum of trivial $D_n'$-modules, as $D_n'$ is the closure of $A_n$ which is simple and has a non-trivial intersection with $P_1$.
 Let us assume that \[\dim V \geqslant n-1 \geqslant 5\]

 We obtain a lift of the projective representations of $P_1$ and $P_2$ into $\GL(V)$ using \cref{Barbara even}; it is immediate that the two lifts agree on each $\epsilon_i \epsilon_j$ with $i,j > 2$, since each of the lifts of such an element is an involution with $(-1)$-eigenspace of dimension 2 and $(+1)$-eigenspace of dimension $\dim V - 2 > 2$, lifting $\pi(\epsilon_i \epsilon_j)$. Similarly, the lifts of the elements $\sigma_{ij} \sigma_{kl}$ (with $i,j,k,l>2$ all distinct) also agree. It follows that the lifts agree on $P_1 \cap P_2 \cong D'_{n-2}$.

We now repeat the argument for any two stabilisers $P_i$ and $P_j$. This way we have defined a map from generators of $D_n'$  to $\GL(V)$, which respects all relations supported by some $P_i$. But it is easy to see that such relations are sufficient for defining the group, and so the map induces a homomorphism $\overline \pi \colon D'_n \to \GL(V)$ as required.

Let $U_i$ denote the standard $P_i$-module, and $W_i$ its complement which is a sum of trivial $P_i$-modules.
Let $U = \sum U_i$. We claim that $U$ is $D'_n$-invariant: take a generator $\xi$ of $D'_n$ lying in, say, $P_1 \s- P_2$. Let $x \in U_2$. Then $x = y + z$ with $y \in U_1$ and $z \in W_1$, and so
\[
 \overline{\pi}(\xi)(x)= y' + z = x -(y - y') \in U_2 + U_1
\]
Similar computations for arbitrary indices prove the claim.
Now \cref{reps of D'n} implies that $U$ is the standard representation of $D'_n$.

Consider $V$ as a $2^{n-1}$-representation. Since $V$ is a vector space over a  field of characteristic $p > 2$, this representation is semi-simple, and so $U$ has a complement $W$. It is clear that $W$ is is a sum of trivial $2^{n-1}$-representations, since all the non-trivial modules of elements $\epsilon_i \epsilon_j$ are contained in some $U_l$, and thus in $U$.
For the same reason it is clear that $W$ is a sum of trivial $A_n$-modules -- for this we look at elements $\sigma_{ij} \sigma_{kl}$. We conclude that $W$ is a sum of trivial $D_n'$-modules.
\end{proof}

\subsection{Projective representations of $\SAut(F_n)$}
Now we use the rigidity of $D_n'$-representations developed above in the context of projective representations of $\SAut(F_n)$. This (together with the information we already have over the field of $3$ elements) will be enough to conclude that for $n\geqslant 10$ the classical groups in odd characteristics are not the smallest quotients of $\SAut(F_n)$.

\begin{thm}
\label{reps odd char}
\label{main thm: representations}
Let $n\geqslant 8$.
 Let $\pi \colon \SAut(F_n) \to \PGL(V)$ be a projective representation of dimension $k$ with $k < 2n-4$ over an algebraically closed field $\Fbar$ of characteristic other than $2$. If $\pi$ does not factor through the natural map $\SAut(F_n) \to \SL_n(\Z)$, then $k \geqslant n+1$ and the projective module $V$ contains a trivial projective module of dimension $k-n-1$.

\end{thm}
\begin{proof}

Since $\Fbar$ is algebraically closed, $\PGL(V) = \PSL(V)$.
As $\pi$ does not factor through the natural map $\SAut(F_n) \to \SL_n(\Z)$, \cref{killing epsilon} tells us that $\pi$ restricted to $D_n'$ is injective.

By \Cref{Barbara} we see that there is a lifting $\overline \pi$ of the projective representation of $D'_n$ to a linear representation on $V$ such that $V = W \oplus U$ as a $D'_n$-module, where $U$ is standard and $W$ is a sum of trivial modules.  Let $u_1, \dots, u_n$ denote the standard basis for $U$.
For notational convenience we will write
 \[
  \overline \xi = \overline \pi(\xi)
 \]
for $\xi \in D'_n$.

Note that \Cref{Barbara} implies that
\[U = \bigoplus_{i \in N} E_{\{i\}}\]
where $E_{\{i\}}$ is spanned by $u_i$.

Let us pick a lift of $\pi(\rho_{12})$ acting linearly on $V$; we will call it $\overline{\rho_{12}}$.
Since $\rho_{12}$ commutes with $\epsilon_i \epsilon_j$ with $i,j>2$, the element $\overline{\rho_{12}}$ permutes the eigenspaces of $\overline{ \epsilon_i \epsilon_j}$. But for a given pair $(i,j)$, the eigenspaces of $\overline{ \epsilon_i \epsilon_j}$ have dimensions $2$ and $\dim V - 2 \geqslant n-2>2$. Thus $\overline{\rho_{12}}$ preserves each eigenspace of $\overline{ \epsilon_i \epsilon_j}$. It follows that
\[
\overline{\rho_{12}} (u_i) \in \langle u_i \rangle = E_{\{i\}}
\]
for each $i>2$.

Let us choose lifts $\overline {\rho_{ij}}$ of $\pi(\rho_{ij})$ for each pair $(i,j)$. By a discussion identical to the one above we see that $\overline {\rho_{ij}}$ preserves $E_{\{l\}}$ for $l \not\in \{i,j\}$.

We may choose $\overline{\rho_{12}}$ so that it fixes $u_3$.
We have
\[
 [\overline{\rho_{14}}^{-1},\overline{\rho_{42}}^{-1}] = \lambda \cdot \overline{\rho_{12}}^{-1}
\]
for some $\lambda \in \Fbar \s- \{0\}$. But clearly $[\overline{\rho_{14}}^{-1},\overline{\rho_{42}}^{-1}](u_3) = u_3$, since both $\overline{\rho_{14}}^{-1}$ and $\overline{\rho_{42}}^{-1}$ preserve $u_3$ up to homothety. Therefore $\lambda = 1$, and thus
\[
 \overline{\rho_{12}}^{-1}(u_i) = u_i
\]
for all $i > 4$. Replacing $4$ by another number greater than $3$ in the calculation above yields the same result for any $i > 3$. Using an analogous argument we may choose each $\overline{\rho_{ij}}$ so that it fixes $u_l$ for all $l \not \in \{i,j\}$.
It follows that conjugating $\overline {\rho_{12}}$ by an element $\overline{\xi}$ (with $\xi \in A_n$) yields an appropriate element $\overline {\rho_{ij}}$, and not just $\overline {\rho_{ij}}$ up to homothety.

We also see that
$\overline{\rho_{12}}$ preserves $Z = W \oplus \langle u_1,u_2\rangle$, as this is the centraliser of
\[\langle \{ \overline{\epsilon_i \epsilon_j}\} \mid i,j > 2 \} \rangle\]

Note that $W$ is a subspace of $Z$ of codimension $2$; therefore $W' = \overline{\rho_{12}}^{-1}(W) \cap W$ is a subspace of $W$ of codimension at most $2$, and so of dimension at least $k-n-2$. Let $x \in W'$ be any vector. Now $\overline{\rho_{12}}(x)$ lies in $W$, and so
\[ \overline{\rho_{12}}(x) = \overline{\sigma_{12} \sigma_{13} \epsilon_3 \epsilon_2}\, \overline{\rho_{12}}(x)
\]
Thus
\[
\overline {\rho_{31}}\, \overline{\rho_{12}}(x) = \overline {\rho_{31}}\,\overline{\sigma_{12} \sigma_{13} \epsilon_3 \epsilon_2} \, \overline{\rho_{12}}(x) = \overline{\sigma_{12} \sigma_{13} \epsilon_3 \epsilon_2}\, \overline{\rho_{12}}^{-1}\overline{\rho_{12}}(x) = \overline{\sigma_{12} \sigma_{13} \epsilon_3 \epsilon_2}(x) = x
\]
 where the last equality follows from the fact that $x \in W$.
Observe that \[\overline{\rho_{12}}.x = \overline{\epsilon_2 \epsilon_3}\,  \overline{\rho_{12}} .x = \overline{\rho_{12}}^{-1} \, \overline{\epsilon_2 \epsilon_3}.x =   \overline{\rho_{12}}^{-1} .x\]
Using a similar argument we show that
\[
\overline {\rho_{31}}^{-1}\overline{\rho_{12}}^{-1}(x)  = x
\]
and so
\[
 [\overline {\rho_{31}}^{-1}, \overline{\rho_{12}}^{-1}](x) = x
\]
But $[\overline{\rho_{31}}^{-1}, \overline{\rho_{12}}^{-1}] = \overline{\rho_{32}}^{-1}$, and so $\overline{\rho_{32}}(x) = x$. Conjugating by elements $\overline \xi$ with $\xi \in D'_n$ we conclude that
\[
 \overline {\rho_{ij}}(x) = x = \overline{\lambda_{ij}}(x)
\]
for every $i$ and $j$. This implies that $W'$ is preserved by $\SAut(F_n)$, and the restricted projective module is trivial.


\smallskip
If $\dim W' = k-n-1$ then we are done. Let us assume that this is not the case, that is that $W'$ is of codimension $2$ in $W$.
Consider the involution $\rho_{12} \epsilon_2 \epsilon_3$. We set
\[
\overline{\rho_{12} \epsilon_2 \epsilon_3} = \overline{\rho_{12}}\, \overline{\epsilon_2 \epsilon_3}
                                                                      \]
Note that this element satisfies
\[
 \overline{\rho_{12} \epsilon_2 \epsilon_3}^2 = \nu I
\]
for some $\nu \in \F^\times$. But $\overline{\rho_{12} \epsilon_2 \epsilon_3}(u_3) = -u_3$ and so $\nu = 1$. Therefore $\overline{\rho_{12} \epsilon_2 \epsilon_3}$ is an involution.

Let $Y = Z/W'$. Note that $\overline{\epsilon_2 \epsilon_3}$ acts on $Y$, and its $(-1)$-eigenspace of dimension exactly $1$, and the $(+1)$-eigenspace of dimension $3$. We also have an action of the involution  $\overline{\rho_{12} \epsilon_2 \epsilon_3}$
on $Y$.

Since $\overline{\rho_{12} \epsilon_2 \epsilon_3}$ acts trivially on $W'$ and its $(-1)$-eigenspace in the complement of $Z$ in $V$ is of dimension $1$, the dimension of its $(-1)$-eigenspace in $Y$ must be odd (here we use the fact that $\pi$ is a map to $\PSL(\F)$).
Thus
there are  at least two linearly independent vectors $v_1$ and $v_2$ lying in the intersection of the $(+1)$-eigenspace of $\overline{\epsilon_2 \epsilon_3}$ and some eigenspace of  $\overline{\rho_{12} \epsilon_2 \epsilon_3}$.

Since $\overline{\rho_{12} \epsilon_2 \epsilon_3}$ is an involution and the characteristic of $\F$ is odd, there exists a complement of $W'$ in $Z$ on which $\overline{\rho_{12} \epsilon_2 \epsilon_3}$ acts as on $Y$. Thus, we have the two vectors corresponding to $v_1$ and $v_2$; we will abuse the notation by calling them $v_1$ and $v_2$ as well.

Since $W'$ lies in the $(+1)$-eigenspace of $\overline{\epsilon_2 \epsilon_3}$, so do $v_1$ and $v_2$. Thus there is a non-zero linear combination $v_3$ of $v_1$ and $v_2$ which lies in $W$, since the codimension of $W$ in the $(+1)$-eigenspace of $\overline{\epsilon_2 \epsilon_3}$ is $1$. Also, $\overline {\epsilon_2\epsilon_3}$ act trivially on this vector, and so we have found another vector in $W$ which is mapped to $W$ by $\overline{\rho_{12}}$. Arguing exactly as before we show that $\langle v_3 \rangle$ is $\SAut(F_n)$ invariant and trivial as a projective module. Hence $W' \oplus \langle v_3 \rangle$ is also $\SAut(F_n)$ invariant, and is trivial as a projective module since $\SAut(F_n)$ is perfect (\cref{perfect}).
\end{proof}

Let us remark here that there do exist  representations of $\SAut(F_n)$ in dimension $n+1$ over any field (over $\Z$ in fact) which do not factor through the natural map $\SAut(F_n) \to \SL_n(\Z)$ -- see~\cite[Proposition 3.2]{BridsonVogtmann2012}.


\begin{thm} \label{oddcharacteristicn10}
 Let $n \geqslant 10$. Then every finite simple classical group of Lie type in odd characteristic which  is a quotient of $\SAut(F_n)$ is larger in order than $\L_n(2)$.
\end{thm}
\begin{proof}
Let $K$ be such a quotient, and suppose that $\vert K \vert \leqslant \vert \L_n(2) \vert$. Let $k$ denote the rank of $K$.

If $K$ is of type $\A$ or $\Atwo$, then \cref{computation orders type A} tells us that $k \leqslant 2n-8$.
If $K$ is of any other classical type, then
\cref{computation orders} tells us that $k \leqslant n-4$.

Let $V$ be the natural projective module of $K$, and let $m$ denote its dimension. Note that $V$ is an irreducible projective $K$-module,
and $m < 2n-6$. Thus \cref{main thm: representations} implies that either the representation
\[
 \SAut(F_n) \to K \to \PGL(V)
\]
factors through the natural map $\SAut(F_n) \to \SL_n(\Z)$, or $m=n+1$. In the former case we must have $K \cong \L_n(p)$ for some prime $p$, as $K$ is simple. But $\L_n(p)$ is larger than $\L_n(2)$ for $p \geqslant 3$.

We may thus assume that $m=n+1$. If $K$ is of type $\A$ or $\Atwo$ then it is immediate that it is too big.

When $n$ is even this means that $K$ is the simple group $\B_{\frac n 2}(q)$. This is larger (in cardinality) than $\L_n(2)$ for every $q > 3$ by \cref{computation small cases}, and so we may assume that $q = 3$. But this is impossible by \cref{char 3 small}.

When $n$ is odd, $K$ is one of the simple groups $\typeC_{\frac {n+1} 2}(q)$, $\D_{\frac {n+1} 2}(q)$ or $\Dtwo_{\frac {n+1} 2}(q)$. \cref{computation small cases} immediately rules out all values of $q$ except for $q=3$, and again we are done by applying \cref{char 3 small}.
\end{proof}

\section{The exceptional groups of Lie type}
In this section we focus on the exceptional groups of Lie type. These are \begin{enumerate}
\item the Suzuki-Ree groups $\Btwo_2(2^{2m+1})$, $\Gtwo_2(3^{2m+1})$, $\typeFtwo_4(2^{2m+1})$ and $\typeFtwo_4(2)'$,
\item the Steinberg groups $\Dthree_4(q)$, $\Etwo_6(q)$ and
\item the exceptional Chevalley groups $\G_2(q)$, $\typeF_4(q)$, $\E_6(q)$, $\E_7(q)$ and $\E_8(q)$.
\end{enumerate}

They are defined for all $q \geqslant 2$, $m \geqslant 0$ and are all simple with the following exceptions: the group $\mathrm{Sz}(2) \cong 5 \rtimes 4$ which is visibly solvable; the group $\Gtwo_2(3)$ whose index $3$ derived subgroup is isomorphic to $\A_1(8)$; the group $\G_2(2)$ whose index $2$ derived subgroup is isomorphic to $\Atwo_2(3)$; and the group $\typeFtwo_4(2)$ whose index $2$ derived subgroup $\typeFtwo_4(2)'$ is simple.

For simplicity, in this section we always use the type symbols to denote the adjoint versions.

We now introduce the following notation: for a group $K$ the \emph{$A$-rank} of $K$ is the largest $n$ such that $K$ contains a copy of the alternating group $A_n$. In particular, we will make use of the bounds on the \arank{} of the exceptional groups given in \cite[Table 10.1]{lieseitz}.

Generally, to show that a group $K$ is not the smallest quotient of some $\SAut(F_n)$ we argue as follows: let $n(K)$ be the smallest integer such that
\[
 \vert \L_{n(K)-1}(2) \vert <  \vert K \vert \leqslant \vert \L_{n(K)}(2) \vert
\]
and assume that we have an epimorphism $\phi \colon \SAut(F_{n}) \to K$ with $n\geqslant n(K)$.
Now we compare $n$ to the $2$-rank, the \drank{}, and \arank{} of $K$. If the $2$-rank is smaller than ${n-1}$ the we use \cref{killing Dn} applied to the subgroup $2^{n-1}$ and conclude that $K$ is in fact a quotient of $\SL_n(\Z)$. But the smallest such quotient is $\L_n(2)$. If the \arank{} of $K$ is smaller than $n+1$, then we use \cref{killing Dn} applied to the subgroup $A_n$ (we observe that if $A_{n+1}$ is not mapped injectively, then neither is $A_n$ for any $n\geqslant 3$). Similarly for the \drank{}.

If the $2$-rank and \arank{} arguments fail, we look at centralisers. If $n\geqslant 5$ and the simple non-abelian factors of every involution in $K$ have already been shown not to be quotients of $\SAut(F_{n-2})$, then we look at $\phi(\epsilon_1 \epsilon_2)$. If this is trivial then we are done by \cref{killing Dn}; otherwise we obtain a map from $\SAut(F_{n-2})$ (which centralises $\epsilon_1 \epsilon_2$) to a group whose simple composition factors are not quotients of $\SAut(F_{n-2})$ (note that the abelian factors are ruled out by the fact that $\SAut(F_{n-2})$ is perfect). Thus $\SAut(F_{n-2})$ lies in the kernel of $\phi$, and so in particular $\phi$ trivialises some transvection. But then it trivialises every transvection since they are all conjugate, and thus $\phi$ is trivial.

If $n\geqslant 5$ we may argue analogously using the element $\gamma$ of order $3$ from \cref{killing t}; if $n \geqslant 2+k$ for $k\geqslant 5$ odd we may argue in an analogous manner using \cref{centralisers}.

\begin{lem} Let $K$ be a finite simple group belonging to one of the following families:\begin{enumerate}
\item the Suzuki groups $\Btwo_2(2^{2m+1})$,
\item the small Ree groups $\Gtwo_2(3^{2m+1})$,
\item the large Ree groups $\typeFtwo_4(2^{2m+1})$, or,
\item the Tits group $\typeFtwo_4(2)'$,
\end{enumerate}
where $m \geqslant 1$ is an integer.
Then $K$ is not the smallest finite non-trivial quotient of $\SAut(F_{n})$.\end{lem}
\begin{proof}
The smallest $K$ among the families considered is isomorphic to $\mathrm{Sz}(8) = \Btwo_2(8)$, and has order greater than $\vert \L_4(2) \vert$; thus $n \geqslant 5$. The order of $\Btwo_2(2^{2m+1})$ is coprime to $3$, and the order of $\Gtwo_2(3^{2m+1})$ is coprime to $5$. Hence it is clear that the simple Suzuki and small Ree groups cannot be quotients of $\SAut(F_n)$ for $n \geqslant 4$, since the alternating group $A_5$ cannot be mapped injectively, and so we may use \cref{killing Dn}.

Now assume that $K$ is $\typeFtwo_4(2^{2m+1})$ or
the Tits group; observe that the smallest member of this family, the Tits group $\typeFtwo_4(2)'$, has order greater than $\L_5(2)$ and so $n\geqslant 6$. But, by \cite[Proposition 2.2]{2f4} we see that the \arank{} of $K$ is 6.
\end{proof}

\begin{lem} \label{g2and3d4} Let $K$ be a finite simple group belonging to one of the following families:\begin{enumerate}
\item the exceptional groups of type $\G_2(q)$, where $q \geqslant 3$, or
\item the exceptional groups of type $\Dthree_4(q)$, where $q \geqslant 2$.
\end{enumerate}
Then, $K$ is not the smallest finite non-trivial quotient of $\SAut(F_{n})$.\end{lem}
\begin{proof}
First, let $K \cong \G_2(q)$. We divide the proof into the case that $q$ is either odd or even. When $q$ is odd the $2$-rank of $K$ is $3$ by \cite[Lemma 2.4]{2g2}, but
\[\vert K \vert \geqslant \vert G_2(3) \vert > \vert \L_4(2) \vert\]  and so $n\geqslant 5$.

When $q \geqslant 4$ is even, $\vert K \vert > \vert \L_5(2)\vert$ but from inspection of the list of maximal subgroups of $K$ (see \cite{coop}) we see that the \arank{} of $K$ is at most $5$.

For $K \cong \Dthree_4(q)$ note that the smallest member of this family is $\Dthree_4(2)$ and has order greater than $\vert \L_5(2) \vert$. The maximal subgroups of $K$ are known (see \cite{3d4}) and we see that the \arank{} of $K$ is $5$.
\end{proof}

For the remaining groups we again split into the odd and even characteristic cases.

\begin{lem} \label{excodd}
Let $K$ be a finite simple exceptional group of type $\typeF_4$, $\E_6$, $\Etwo_6$, $\E_7$ or $\E_8$ in odd characteristic. Then $K$ is not the smallest non-trivial finite quotient of $\SAut(F_n)$.
\end{lem}
\begin{proof}
It is easy to see that if $K$ belongs to any of these families, then the order of $\vert K \vert$ is bounded below when $q=3$. If $K \cong \typeF_4(q)$, $\E_6(q)$ or $\Etwo_6(q)$, then the \arank{} of $K$ is at most $7$ \cite[Table 10.1]{lieseitz} but the order of $K$ is bounded below by the order of $\typeF_4(3)$ which has order greater than $\L_9(2)$.

If $K \cong \E_7(q)$, then then the \arank{} of $K$ is at most $10$, but the smallest member of this family $\E_7(3)$ has order greater than $\vert \L_{14}(2) \vert$.

Finally, if $K \cong \E_8(q)$, then the \arank{} of $K$ is at most $11$, but the smallest member of this family $\E_8(3)$ has order greater than $\vert \L_{19}(2) \vert$.\end{proof}

\begin{lem} \label{exceven}
Let $K$ be a finite simple exceptional group of type $\typeF_4$, $\E_6$, $\Etwo_6$, $\E_7$ or $\E_8$ defined over a finite field of order $q=2^m \geqslant 4$. Then $K$ is not the smallest non-trivial finite quotient of $\SAut(F_n)$.
\end{lem}
\begin{proof}
It is easy to see that if $K$ belongs to any of these families, then the order of $\vert K \vert$ is bounded below when $q=4$. The degree of the largest alternating group in each of these groups can be found in \cite[Table 10.1]{lieseitz}. If $K \cong \typeF_4(q)$, then the \arank{} of $K$ is $10$, but $K$ has order greater than $\L_{10}(2)$. If $K \cong \E_6(q)$ or $\Etwo_6(q)$, then the \arank{} is bounded above by $12$, but the smallest such group $\E_6(4)$ has order greater than $\L_{12}(2)$. Finally, if $K \cong \E_7(q)$ or $\E_8(q)$, then the \arank{} of $K$ is at most $17$, but the smallest member of this family $\E_7(4)$ has order greater than $\L_{16}(2)$.\end{proof}

In order to dispose of the remaining five cases, we state the following result whose proof can be found in \cite[Sections 15-17]{AschSeitz}.

\begin{lem} \begin{enumerate}
\item Any non-abelian composition factor of an involution centraliser in $\E_6(2)$ is isomorphic to one of $\A_2(2)$, $\A_5(2)$ or $\B_3(2)$.
\item Any non-abelian composition factor of an involution centraliser in $\E_7(2)$ is isomorphic to one of $\B_3(2)$, $\B_4(2)$, $\D_6(2)$ or $\typeF_4(2)$.
\item Any non-abelian composition factor of an involution centraliser in $\E_8(2)$ is isomorphic to one of $\B_4(2)$, $\B_6(2)$, $\typeF_4(2)$ or $\E_7(2)$.
\end{enumerate} \end{lem}

We are now in a position to prove the following.

\begin{lem} \label{excfield2}
Let $K$ be a finite simple exceptional group of type $\typeF_4(2)$, $\E_6(2)$, $\Etwo_6(2)$, $\E_7(2)$ or $\E_8(2)$. Then $K$ is not the smallest non-trivial finite quotient of $\SAut(F_n)$.\end{lem}
\begin{proof} 
If $K \cong \typeF_4(2)$, then $n\geqslant 8$, but from the comparison of the character tables of $K$ \cite{atlas} and of $D_8'$ which can be performed in GAP, we see that $D_8'$ is not a subgroup of $K$, and we use \cref{killing Dn}. We eliminate the case $K \cong \Etwo_6(2)$ in the same way, except that here $n\geqslant 9$.

If $K \cong \E_6(2)$, then $\vert K \vert > \vert \L_8(2) \vert$. By the preceding lemma, it remains to show that any homomorphism from $\SAut(F_n)$ with $n\geqslant 7$ to $\A_2(2)$, $\A_5(2)$ or $\B_3(2)$ is trivial. It can easily be checked in GAP that none of these groups contains a subgroup isomorphic to $D_7'$, hence the result follows from \cref{killing Dn}. (Also, we will revisit $\SAut(F_7)$ in the next section.)

If $K \cong \E_7(2)$, then $\vert K \vert > \vert \L_{11}(2) \vert$. By the preceding lemma, it remains to show that any homomorphism from $\SAut(F_{n})$ with $n\geqslant 10$ to $\B_3(2)$, $\B_4(2)$, $\D_6(2)$ or $\typeF_4(2)$ is trivial. The groups $\B_3(2)$, $\B_4(2)$ and $\D_6(2)$ are of classical type in even characteristic and smaller in cardinality than $\L_{10}(2)$, hence we can apply \cref{thm char 2}.
The maximal subgroups of $\typeF_4(2)$ are known and can be found in \cite{atlas}; it is clear by inspection that the \arank{} of $\typeF_4(2)$ is $10$.

Finally, if $K \cong \E_8(2)$, then $\vert K \vert > \vert \L_{15}(2) \vert$. By the preceding lemma, it remains to show that any homomorphism from $\SAut(F_{n})$ with $n\geqslant 14$ to $\B_4(2)$, $\B_6(2)$, $\typeF_4(2)$ or $\E_7(2)$ is trivial. As before, \cref{thm char 2} takes care of $\B_4(2)$ and $\B_6(2)$ since they are smaller in cardinality than $\L_{14}(2)$, whereas the \arank{} of $\typeF_4(2)$ and $\E_7(2)$ is at most $13$. This completes the proof.\end{proof}

We can now summarise the preceding lemmata.

\begin{thm} \label{noexceptionals} Let $K$ be a finite simple group of exceptional type. If $K$ is a quotient of $\SAut(F_n)$, then $\vert K \vert > \vert \L_n(2) \vert$. \end{thm}

In fact, using the \arank{} we can say more: when $n>16$ then the exceptional groups of Lie type are never quotients of $\SAut(F_n)$, see \cite{lieseitz}.

\section{Small values of $n$ and the conclusion}

We can now conclude the paper.

\begin{thm}
 \label{very main thm}
 Let $n\geqslant 3$. Every non-trivial finite quotient of $\SAut(F_n)$ is either greater in cardinality than $\L_n(2)$, or isomorphic to $\L_n(2)$. Moreover,  if the quotient is $\L_n(2)$, then the quotient map is the natural map postcomposed with an automorphism of $\L_n(2)$.
\end{thm}
\begin{proof}
 Suppose that $n \geqslant 8$, and let $K$ be a smallest non-abelian quotient of $\SAut(F_n)$. Since $\SAut(F_n)$ is perfect, $K$ is simple. By
\cref{no alternating quotients}, $K$ is not an alternating group; by \cref{main: spo}, $K$ is not a sporadic group; by \cref{oddcharacteristicn10}, $K$ is not a classical group of Lie type in odd characteristic; by \cref{noexceptionals}, $K$ is not an exceptional group of Lie type. Finally, by \cref{thm char 2}, $K$ is isomorphic to $\L_n(2)$, and the quotient map is obtained by postcomposing the natural map $\SAut(F_n) \to \L_n(2)$ by an automorphism of $\L_n(2)$.

For $3 \leqslant n < 8$, the result follows from \cref{smalln3,smalln4,smalln5,smalln6,smalln7} below.
\end{proof}

As indicated above, we now verify \cref{very main thm} for $n \in \{3,\dots,7\}$.
Note that in view of \cref{uniqueness}, it is enough to show that a smallest quotient of $\SAut(F_n)$ is isomorphic to $\L_n(2)$.

By \cref{no alternating quotients,main: spo,noexceptionals} we can assume that $K$ is of classical type. We make use of the list of simple groups in order of size appearing in \cite[pgs. 239--242]{atlas}. Note that this list does not contain all members of the families of types $\A_1(q)$, $\A_2(q)$, $\Atwo_2(q)$, $\A_3(q)$, $\typeC_2(q)$ or $\G_2(q)$; we can exclude $\G_2(q)$ by \cref{g2and3d4} and the following lemma allows us to take care of the rest.

\begin{lem}[{\cite[Theorem 4.10.5]{Gorensteinetal1998}}] \label{gls2rank} Let $K \leqslant \A_n(q)$ where $q$ is odd. If $n \leqslant 3$, then the $2$-rank of $K$ is bounded above by $n+1$.\end{lem}


The general strategy is exactly as described in the previous section. As before, we use types to denote the adjoint versions.

We now look at each value of $n$ separately.

\begin{lem}[$n=3$]
\label{smalln3}
Let $K$ be a non-abelian finite simple group with $\vert K \vert \leqslant \vert \L_3(2) \vert$. If $K$ is a quotient of $\SAut(F_3)$, then $K \cong \L_3(2)$.\end{lem}
\begin{proof}
If $K$ is a non-abelian simple group not isomorphic to $\L_3(2)$ and order at most $\vert \L_3(2) \vert$, then $K \cong A_5$. But $A_5$ is not a quotient of $\SAut(F_3)$ by \cref{A_5 not quotient of F_3}.
\end{proof}

\begin{lem}[$n=4$]
\label{smalln4}
Let $K$ be a non-abelian finite simple group with $\vert K \vert \leqslant \vert \L_4(2) \vert$. If $K$ is a quotient of $\SAut(F_4)$, then $K \cong \L_4(2)$.
\end{lem}
\begin{proof}
Let $K$ be a simple group of order at most $\vert \L_4(2) \vert$ and a quotient of $\SAut(F_4)$. Assume that $K$ is not isomorphic to $\L_4(2)$. By \cref{gls2rank}, $K$ is not a subgroup of $\A_2(q)$ where $q$ is odd. Hence, $K$ is isomorphic to one of the following.
\[\A_1(8), \; \A_1(16), \; \A_2(4)\]
With the exception of $\A_2(4)$ (which has the same order as $\L_4(2)$), it is clear from the inspection of their maximal subgroups \cite{atlas} that they do not contain subgroups isomorphic to $D_4'$. In the case of $\A_2(4)$, there is a subgroup isomorphic to $2^4 \rtimes A_5$, however this is not isomorphic to the group $D_5'$. It can be computed in GAP that $\A_2(4)$ does not contain subgroups isomorphic to $D_4'$. This completes the proof.
\end{proof}

\begin{lem}[$n=5$]
\label{smalln5}
Let $K$ be a non-abelian finite simple group with $\vert K \vert \leqslant \vert \L_5(2) \vert$. If $K$ is a quotient of $\SAut(F_5)$, then $K \cong \L_5(2)$.
\end{lem}
\begin{proof}
Assume that $K$ is not isomorphic to $\L_5(2)$. By \cref{gls2rank,kleidmanliebeck} we can exclude all but the following groups.
\[\A_2(4), \; \Atwo_2(4), \; \Atwo_2(8), \; \A_3(3), \; \Atwo_3(2) \cong \typeC_2(3), \; \typeC_2(4), \; \typeC_2(5), \; \Atwo_3(3), \; \typeC_3(2)\]

The groups $\A_3(3)$ and $\Atwo_3(3)$ are dealt with in \cref{char 3 odd base case}. Excluding those which also do not contain $D_5'$ as subgroups we are left with the possibilities $\typeC_3(2)$ and $\typeC_2(5)$. If $K \cong \typeC_2(5)$ or $\typeC_3(2)$, then any non-abelian composition factor of an involution centraliser is isomorphic to $A_5$ or $A_6$, neither of which is a quotient of $\SAut(F_3)$ by \cref{A_5 not quotient of F_3}, a contradiction.\end{proof}

\begin{lem}[$n=6$]
\label{smalln6}
Let $K$ be a non-abelian finite simple group with $\vert K \vert \leqslant \vert \L_6(2) \vert$. If $K$ is a quotient of $\SAut(F_6)$, then $K \cong \L_6(2)$.
\end{lem}
\begin{proof}
Assume that $K$ is not isomorphic to $\L_6(2)$. By \cref{gls2rank,kleidmanliebeck} we can assume that $K$ has dimension at least $4$ in even characteristic, in order to contain a subgroup isomorphic to $A_7$, and dimension at least $5$ in odd characteristic in order for the $2$-rank to be at least $5$. Hence $K$ is isomorphic to one of the following:
\[\typeC_2(8), \; \A_3(4), \; \Atwo_3(4), \; \typeC_3(3), \; \typeC_3(3), \; \Atwo_4(2), \; \D_4(2), \; \Dtwo_4(2), \; \Atwo_5(2)\]
Those groups which contain subgroups isomorphic to $D_6'$ are isomorphic to $\Atwo_5(2)$, $\B_3(3)$, $\D_4(2)$ and $\Dtwo_4(2)$. The simple factors of the centralisers of elements of order $3$ in $\Atwo_5(2)$ and $\B_3(3)$ can be computed in GAP and are isomorphic to $\A_1(9)$ or $\typeC_2(3)$, neither of which is a quotient of $\SAut(F_4)$ -- this follows from \cref{char 3 even base case} for $\typeC_2(3)$ and from \cref{smalln4} for $\A_1(9)$, since they are smaller in cardinality than $\L_4(2)$.

The simple factors of the involution centralisers of $\D_4(2)$ and $\Dtwo_4(2)$ are isomorphic to $\A_1(4)$ or $\A_1(9)$, neither of which is a quotient of $\SAut(F_4)$, by \cref{smalln4}, since they are both smaller in cardinality than $\L_4(2)$.
\end{proof}

\begin{lem}[$n=7$]
\label{smalln7}
Let $K$ be a non-abelian finite simple group with $\vert K \vert \leqslant \vert \L_7(2) \vert$. If $K$ is a quotient of $\SAut(F_7)$, then $K \cong \L_7(2)$.
\end{lem}
\begin{proof}
Assume that $K$ is not isomorphic to $\L_7(2)$. Again we make use of the values listed in \cref{kleidmanliebeck} for $R_p(A_8)$, hence we need to consider groups of dimension at least $7$ in odd characteristic and dimension at least $4$ in even characteristic.

In even characteristic we are left with the following
\[\A_3(8), \; \Atwo_3(8), \; \Atwo_4(4), \; \B_2(16), \; \typeC_3(4), \; \typeC_4(2),\; \D_5(2), \; \Dtwo_5(2)\]
none of which is a quotient of $\SAut(F_7)$ by \cref{main thm char 2 lie type}, with the exception of $\Dtwo_5(2)$.

Now let $K \cong \D_5(2)$. The non-abelian simple quotients of the involution centralisers in $K$ are isomorphic to $A_6$, $A_8$ or $\typeC_3(2)$. Since all of these are smaller than $\L_5(2)$ we apply Lemma \ref{smalln5} to rule them out.

In odd characteristic we have the groups
\[\D_4(3), \; \Dtwo_4(3)\]
Any non-abelian simple  factor of a centraliser of an element of order $3$ in either of these groups is isomorphic to $\A_1(9)$ or to $\typeC_2(3)$. By \cref{smalln5}, neither of these groups is a quotient of $\SAut(F_5)$ since they are smaller in cardinality than $\L_5(2)$, which completes the proof.
\end{proof}

\appendix
\section{Computations}

This appendix contains all the necessary computations. Note that we use type symbols to denote the adjoint versions of the groups of Lie type.

\begin{lem}
\label{computation actions 2}
 For $n\geqslant 8$ we have
 \[
  2^{n-3} > \binom n 2
 \]
\end{lem}
\begin{proof}
 It is enough to observe that the result is true for $n=8$, and
 \[
  \dfrac {\binom {n+1} 2}{\binom n 2} = \dfrac {n+1} {n-1} \leqslant 2
 \]
for all $n\geqslant 3$.
\end{proof}

\begin{lem}
\label{computation actions}
For an even $n\geqslant 12$ we have
\[
\frac 1 2 \binom n {\frac n 2} \geqslant \min \Big\{ \binom {n} {\lfloor \frac {n} 4 \rfloor}, 2^{n-\lfloor \frac {n} 4 \rfloor -1}\Big\}
\]
\end{lem}
\begin{proof}
Let $n=2m$. We have
\begin{align*}
\dfrac{\frac 1 2 \binom n {\frac n 2}}{ \binom {n} {\lfloor \frac {n} 4 \rfloor}} &= \dfrac {(\lfloor \frac m 2 \rfloor)!(2m-\lfloor \frac m 2 \rfloor)! }{2\cdot m! m!} \\
 &= \dfrac 1 2 \cdot\prod_{i=1}^{\lceil \frac m 2 \rceil} \dfrac{m+i}{\lfloor \frac m 2 \rfloor+i}\\
  &= \dfrac {(m+1)(m+2)}{2(\lfloor \frac m 2 \rfloor+1)(\lfloor \frac m 2 \rfloor+2)}  \cdot\prod_{i=3}^{\lceil \frac m 2 \rceil} \dfrac{m+i}{\lfloor \frac m 2 \rfloor+i} \\
  &\geqslant \dfrac {(m+1)(m+2)}{(m +2)( \frac m 2 +2)} \\
   &\geqslant \dfrac {2m+2}{m+4} \\
   &\geqslant 1
\end{align*}
for any $m\geqslant 2$.

We also have
\[
 2^{n-\lfloor \frac {n} 4 \rfloor-1} \leqslant 2^{n- \frac {n} 4 -\frac 1 2} = 2^{\frac {3m-1} 2 }
\]
and
\[
 \dfrac {2^{\frac {3(m+1)-1} 2 }}{2^{\frac {3m-1} 2 }} = 2^{\frac 3 2} < 3
\]
Now
\[
\dfrac{\frac 1 2 \binom {n+2} {\frac {n+2} 2}}{\frac 1 2 \binom n {\frac n 2}} = \dfrac {(2m+1)(2m+2) }{(m+1)^2} \geqslant 3
\]
We conclude by remarking that $\frac 1 2 \binom n {\frac n 2} \geqslant 2^{\frac {3n-2} 4}$ for $n=12$.
\end{proof}

\begin{lem}
\label{computation alternaing}
For $n \geqslant 7$ we have
$\binom n 2 ! \cdot \frac 1 2 > \vert \L_n(2) \vert$.
\end{lem}
\begin{proof}
We have
\[
2^{n^2}> \vert \L_n(2) \vert
\]
since the left-hand side is the number of $n \times n$ matrices over the field of 2 elements. We also have
\[
m ! \geqslant (2 \pi m)^{\frac 1 2} (\frac m e)^m
\]
by Stirling's approximation. Putting $m = \binom n 2 = \frac {n(n-1)} 2$ we obtain
\begin{align*}
 \binom n 2 ! \cdot \frac 1 2 &\geqslant \frac 1 2 ( \pi n(n-1))^{\frac 1 2}   (\frac {n(n-1)} {2e})^{\frac {n(n-1)} 2} \\
&= \frac {\pi^{\frac 1 2}} 2 \cdot \frac{(n(n-1))^{\frac {n^2-n+1} 2}}{(2e)^{\frac {n(n-1)} 2}} \\
&\geqslant \frac 1 2 \cdot \frac{2^{\frac {5(n^2-n+1)} 2}}{2^{\frac {5n(n-1)} 4}} \\
&= 2^{\frac {10(n^2-n+1) -5n(n-1)  -4} 4} \\
&= 2^{\frac {5 n^2-5n+6} 4}
\end{align*}
where we have used the fact that $2^{\frac 5 2} > 2e$ and that $n(n-1)\geqslant 2^5$, as $n \geqslant 7$.

Now
\[ 5 n^2-5n+6 > 4 n^2 \]
holds for every $n \geqslant 4$ and we are done.
\end{proof}

We will now proceed to compute certain inequalities between orders of adjoint versions of finite groups of Lie type -- these orders can be found in \cite[pg. xvi]{atlas}. Let us start by some general remarks.

Firstly, if we fix the type, rank and characteristic, then enlarging the field always results in enlarging the group: this is obvious for the universal versions, and for adjoint versions requires comparing the sizes of centres of the universal versions; such a comparison can easily be performed. Since we will be looking at the smallest groups of a given type, rank and characteristic, we may therefore assume that the field is of prime cardinality.

In fact, arguing as above, we see that for odd characteristics we may assume that the field is of size $3$, and for odd characteristics greater than $3$ we may assume the field to be of size $5$.

Secondly, if we fix the type and field, then increasing the rank always results in  enlarging the group. The argument is precisely as above. The same holds for twisted rank, since to increase the twisted rank we have to increase the rank.

\begin{lem}
\label{computation orders char 2}
Let $n\geqslant 8$. Then every finite group of Lie type in characteristic $2$ of twisted rank at least $n-2$ is larger than $\L_n(2)$, with the exception of $\A_{n-2}(2)$ and $\A_{n-1}(2)$.
\end{lem}
\begin{proof}
By the discussion above, it is enough to prove the result for
 \[
  \A_{2n-2}(4), \Atwo_{2n-3}(2), \B_{n-2}(2), \typeC_{n-2}(2), \D_{n-2}(2), \Dtwo_{n-1}(2)
 \]
and $\E_6(2), \E_7(2)$ and $\E_8(2)$ for small values of $n$.

For the groups of type $\E$ we confirm the result by a direct computation.

The orders of $\B_{n-2}(2)$ and $\typeC_{n-2}(2)$ are equal, and for all $n \geqslant 1$ we have the following identities
\[\frac{\vert \B_n(2) \vert }{2^n(2^n+1)} = \vert \D_n(2) \vert = \frac{\vert \Dtwo_{n+1}(2) \vert}{2^{2n}(2^{n+1}+1)(2^n+1)}\]
Furthermore, $\D_{n-2}(2)$ is a subgroup of $\A_{2n-5}(2)$, and  $\vert \A_{2n-5}(2) \vert < \vert \Atwo_{2n-4}(2) \vert$ when $n \geqslant 4$. We also have
$\vert \A_{2n-5}(2) \vert < \vert \Atwo_{2n-2}(4) \vert$
Therefore, $\D_{n-2}(2)$ is the smallest group we are considering, and so it remains to prove that $\vert \D_{n-2}(2) \vert > \vert \A_{n-1}(2) \vert$.

\begin{align*}
 \frac{\vert \D_{n-2}(2) \vert}{\vert \A_{n-1}(2) \vert} &= \frac{2^{(n-2)(n-3)}(2^{n-2}-1) \prod_{i=1}^{n-3}(2^{2i}-1)}{2^{n(n-1)/2} \prod_{i=1}^{n-1}(2^{i+1}-1)} \\
&=\frac{2^{(n-2)(n-3)}(2^{n-2}-1) \prod_{i=1}^{n-3}(2^i-1)\prod_{i=1}^{n-3}(2^i+1)}{2^{n(n-1)/2} \prod_{i=1}^{n-1}(2^{i+1}-1)} \\
&= 2^{\frac{n^2-9n+12}{2}} \frac{\prod_{i=1}^{n-3}(2^i+1)}{(2^n-1)(2^{n-1}-1)}\\
&> 2^{\frac{n^2-9n+12}{2}} \frac{2^{(n-2)(n-3)/2}}{2^{2n-1}}\\
&= 2^{\frac{n^2-9n+12}{2}} 2^\frac{(n^2-9n+8)}{2} \\
&= 2^{n^2-9n+10}
\end{align*}
which is at least $1$ for all $n \geqslant 8$.
\end{proof}

\begin{lem}
\label{computation orders type A}
Let $K$ be any version of a finite classical group of type $\A_k$ or $\Atwo_k$ in odd characteristic. For every $n \geqslant 6$, if $k \geqslant 2n-7$ then $\vert K \vert > \vert \L_n(2)\vert$.
\end{lem}
\begin{proof}
By the previous discussion, it is clear that it is enough to consider the smallest rank, that is $k = 2n-7$, and the simple group $K$. Also, it is enough to consider $q=3$, as the orders increase with the field -- this is obvious for the universal versions, and for the simple groups follows from inspecting the sizes of the centres of the universal versions.

We have $\vert \Atwo_k(3) \vert \geqslant \frac 1 2 \vert \A_k(3) \vert$, and
\begin{align*}
\frac 1 {2} \vert \A_{2n-7}(3) \vert
&\geqslant \frac 1 {4} \cdot 3^{\binom {2n-6} 2} \cdot \prod_{i=1}^{2n-7} (3^{i+1}-1) \\
&\geqslant 2^{-2} \cdot 2^{\frac{3(2n-6)(2n-7)}{4}} \cdot \prod_{i=1}^{2n-7} 2^{\frac {3i}{2}} \\
&= 2^{\frac{-8 + 3(2n-6)(2n-7) + 3(2n-7)(2n-6)}{4}}\\
&= 2^{6n^2 -39n + 61 }\\
&= 2^{\binom n 2} \cdot 2^{\frac{11n^2 -77n + 122}{2} }\\
&= 2^{\binom n 2} \cdot \prod_{i=1}^{n} 2^{i+1} \cdot 2^{5n^2 -40n + 60 }\\
&> 2^{\binom n 2} \cdot \prod_{i=1}^{n} (2^{i+1}-1) \\
&= \vert \A_{n-1}(2) \vert
\end{align*}
where the last inequality holds for $n\geqslant 6$.
\end{proof}

\begin{lem}
\label{computation orders}
Let $n\geqslant 8$, and let $K$ be any version of a finite classical group of of type $\B_k$, $\typeC_k$, $\D_k$ or $\Dtwo_k$ in odd characteristic. If $k \geqslant n-3$ then $\vert K \vert > \vert \L_n(2)\vert$.
\end{lem}
\begin{proof}
By the previous discussion we take $k = n-4$, $q=3$, and the adjoint version $K$.

Note that $\vert \B_k(3) \vert = \vert \typeC_k(3) \vert$; also $\vert \B_k(3) \vert > \vert \D_k(3) \vert$.
We also have
$\vert \Dtwo_k(3) \vert  \geqslant \frac 1 2 \cdot \vert \D_k(3) \vert$. We then have
\begin{align*}
\frac 1 2 \cdot \vert \D_{n-3}(3) \vert
&\geqslant \frac 1 8 \cdot  3^{(n-3)(n-4)}\cdot (3^{n-3}-1)\cdot \prod_{i=1}^{n-4}(3^{2i}-1)\\
&> 2^{-3} \cdot  2^{\frac{3(n-3)(n-4)}{2}} \cdot 2^{\frac{3(n-3)}{2}} \cdot \prod_{i=1}^{n-4}2^{3i}\\
&= 2^{\frac 3 2 (-2 + (n-3)(n-4) + n-3 + (n-3)(n-4))}\\
&= 2^{\frac 3 2 (2n^2 -13n +19)}\\
&= 2^{\binom n 2} 2^{\frac 1 2 (5n^2 -38n +19)}\\
&= 2^{\binom n 2} \cdot \prod_{i=1}^{n} 2^{i+1} \cdot 2^{\frac 1 2 (4n^2 -41 n +17) }\\
&> 2^{\binom n 2} \cdot \prod_{i=1}^{n} (2^{i+1}-1) \\
&= \vert \A_{n-1}(2) \vert
\end{align*}
where the last inequality holds for $n\geqslant 10$. In the cases $n=8$ or $9$, it can be verified directly that our claim holds.
\end{proof}

\begin{lem}
\label{computation small cases}
For $n\geqslant 4$ and $q>3$ odd, the simple groups $\B_{\frac n 2}(q)$ (when $n$ is even), $\typeC_{\frac {n+1} 2}(q)$, $\D_{\frac {n+1} 2}(q)$ and $\Dtwo_{\frac {n+1} 2}(q)$ (when $n$ is odd) are larger in cardinality than $\L_n(2)$.
\end{lem}
\begin{proof}
When $n$ is odd, all of the orders are bounded below by the order of $\D_{\frac {n+1} 2}(5)$, which is
\begin{align*}
 5^{\frac {n^2-1} 4 } \cdot \prod_{i=1}^{\frac {n-1} 2} (5^{2i}-1) \cdot (5^{\frac{n+1}2} - 1) \cdot \frac 1 4  &\geqslant  2^{\frac{n^2-1} 2} \cdot \prod_{i=1}^{\frac {n-1} 2} 2^{4i} \cdot 2^{n+1} \cdot \frac 1 4 \\
 & = 2^{\frac{n^2-1} 2 + \frac{n^2-n} 2 +n -1} \\
 & = 2^{n^2 + \frac n 2 -\frac 1 2} \\
 &> 2^{n^2} \\
 &= 2^{\binom n 2} \cdot 2^{\binom{n+1} 2} \\
 &> 2^{\binom n 2} \cdot \prod_{i=1}^n (2^i-1) \\
 & = \vert \L_n(2) \vert
\end{align*}

When $n$ is even, we have
\begin{align*}
\vert \B_{\frac n 2}(5) \vert &= 5^{\frac {n^2} 4} \prod_{i=1}^{\frac n 2} (5^{2i} -1 )\\
&> 2^{\frac {n^2} 2} \prod_{i=1}^{\frac n 2} 2^{4i} \\
&= 2^{\frac 1 2 (n^2 + n(n+2))} \\
&= 2^{n^2 + 1} \\
&> \vert L_n(2) \vert \qedhere
\end{align*}
\end{proof}

\bibliographystyle{math}
\bibliography{bibliographyout}

\bigskip
\noindent Barbara Baumeister \hfill \texttt{b.baumeister@math.uni-bielefeld.de} \newline
\noindent Dawid Kielak \hfill \texttt{dkielak@math.uni-bielefeld.de} \newline
\noindent Emilio Pierro \hfill \texttt{e.pierro@mail.bbk.ac.uk} \newline
Fakult\"at f\"ur Mathematik  \newline
Universit\"at Bielefeld \newline
Postfach 100131  \newline
D-33501 Bielefeld \newline
Germany \newline

\end{document}